\documentclass[11pt,twoside,reqno]{amsart}
\usepackage{geometry}
\usepackage[latin1]{inputenc}
\usepackage[OT1]{fontenc}
\usepackage{amsmath}
\usepackage{amsthm}
\usepackage{amssymb}
\usepackage[all]{xy}
\usepackage{bbm}
\usepackage{amscd}
\usepackage{a4wide}
\usepackage{enumitem}

\setitemize{itemsep=0.1cm}

\usepackage{stmaryrd}
\usepackage{mathtools}

\setenumerate[0]{label=(\roman*)}

\DeclareMathOperator{\Spec}{\mathsf{Spec}}

\DeclareMathOperator{\id}{\mathsf{id}}

\DeclareMathOperator{\Hom}{\mathsf{Hom}}

\DeclareMathOperator{\Ext}{\mathsf{Ext}}

\DeclareMathOperator{\Coh}{\mathsf{Coh}}

\DeclareMathOperator{\Ho}{\mathsf H}

\DeclareMathOperator{\GL}{GL}

\DeclareMathOperator{\supp}{\mathsf{supp}}

\DeclareMathOperator{\triv}{\mathsf{triv}}

\DeclareMathOperator{\Z}{\mathsf Z}

\DeclareMathOperator{\ord}{ord}

\DeclareMathOperator{\Ind}{\mathsf{Ind}}

\DeclareMathOperator{\rep}{\mathsf{rep}}

\let\ord\relax
\DeclareMathOperator{\ord}{\mathsf{ord}}

\let\coker\relax
\DeclareMathOperator{\coker}{\mathsf{coker}}

\let\det\relax
\DeclareMathOperator{\det}{\mathsf{det}}

\newcommand{\D}{{\mathsf D}}

\DeclareMathOperator{\Aut}{\mathsf{Aut}}

\newcommand{\cH}{{\mathcal H}}

\newcommand{\cU}{{\mathcal U}}

\newcommand{\hZ}{\widehat Z}

\newcommand{\IA}{\mathbb{A}}
\newcommand{\IC}{\mathbb{C}}

\newcommand{\IK}{\mathbb{K}}

\newcommand{\IQ}{\mathbb{Q}}

\newcommand{\IZ}{\mathbb{Z}}

\DeclareMathOperator{\conj}{\mathsf{conj}}
\DeclareMathOperator{\Ann}{\mathsf{Ann}}

\DeclareMathOperator{\Res}{\mathsf{Res}}

\DeclareMathOperator{\FM}{\mathsf{FM}}

\DeclareMathOperator{\CC}{\mathsf{C}}

\DeclareMathOperator{\irr}{\mathsf{irr}}

\DeclareMathOperator{\Van}{\mathsf{Van}}

\makeatletter
\newcommand{\leqnomode}{\tagsleft@true}
\newcommand{\reqnomode}{\tagsleft@false}
\makeatother

\let\ker\relax
\DeclareMathOperator{\ker}{\mathsf{ker}}

\let\dim\relax
\DeclareMathOperator{\dim}{\mathsf{dim}}

\let\min\relax
\DeclareMathOperator{\min}{\mathsf{min}}

\let\max\relax
\DeclareMathOperator{\max}{\mathsf{max}}

\let\GL\relax
\DeclareMathOperator{\GL}{\mathsf{GL}}

\newcommand{\sym}{\mathfrak S}

\newcommand{\cA}{\mathcal A}
\newcommand{\cB}{\mathcal B}

\newcommand{\cC}{\mathcal C}

\newcommand{\tzeta}{\tilde\zeta}

\newcommand{\cT}{\mathcal T}

\newcommand{\cS}{\mathcal S}

\newcommand{\fm}{\mathfrak m}
\newcommand{\alt}{\mathfrak a}

\newcommand{\reg}{\mathcal O}

\renewcommand{\theta}{\vartheta}
\renewcommand{\phi}{\varphi}
\renewcommand{\_}{\underline{\,\,\,\,}}

\usepackage{lscape}
\usepackage{adjustbox}

\usepackage[usenames,dvipsnames]{xcolor}
\usepackage{hyperref}
\hypersetup{
    colorlinks,
    linkcolor={red!50!black},
    citecolor={blue!50!black},
    urlcolor={blue!80!black}
}
\usepackage{aliascnt} 

\usepackage{tikz-cd}

\newtheorem{theorem}{Theorem}[section]

\newaliascnt{conjecture}{theorem}
\newtheorem{conjecture}[conjecture]{Conjecture}
\aliascntresetthe{conjecture}

  \newaliascnt{proposition}{theorem}
  \newtheorem{prop}[proposition]{Proposition}
  \aliascntresetthe{proposition}

  \newaliascnt{lemma}{theorem}
  \newtheorem{lemma}[lemma]{Lemma}
  \aliascntresetthe{lemma}

  \newaliascnt{corollary}{theorem}
  \newtheorem{cor}[corollary]{Corollary}
  \aliascntresetthe{corollary}

\theoremstyle{definition}

  \newaliascnt{definition}{theorem}
  \newtheorem{definition}[definition]{Definition}
  \aliascntresetthe{definition}

  \newaliascnt{remark}{theorem}
  \newtheorem{remark}[remark]{Remark}
  \aliascntresetthe{remark}

  \newaliascnt{condition}{theorem}
  
  \aliascntresetthe{condition}

\newaliascnt{convention}{theorem}
 
\aliascntresetthe{convention}

\newaliascnt{assumption}{theorem}
 \newtheorem{assumption}[assumption]{Assumption}
\aliascntresetthe{assumption}

  \newaliascnt{question}{theorem}
  
  \aliascntresetthe{question}

  \newaliascnt{example}{theorem}
  
  \aliascntresetthe{example}

\definecolor{viridislightgreen}{RGB}{94, 201, 98}
\definecolor{viridisdarkgreen}{RGB}{33, 145, 140}
\definecolor{viridisblue}{RGB}{59, 82, 139}
\definecolor{viridisviolet}{RGB}{68, 1, 84}

\tikzcdset{scale cd/.style={every label/.append style={scale=#1},
    cells={nodes={scale=#1}}}}

\begin{document}

 \title{Semi-Orthogonal Decompositions for Rank Two Imprimitive Reflection Groups}
 \author[A.\ Krug]{Andreas Krug}
\address{
Institut f\"ur algebraische Geometrie,
Gottfried Wilhelm Leibniz Universit\"at Hannover,
Welfengarten 1,
30167 Hannover,
Germany
}
\email{krug@math.uni-hannover.de}

  \maketitle
\begin{abstract}
For every imprimitive complex reflection group of rank 2, we construct a semi-orthogonal decomposition of the derived category of the associated global quotient stack which categorifies the usual decomposition of the orbifold cohomology indexed by conjugacy classes. This confirms a conjecture of Polishchuk and Van den Bergh in these cases. This conjecture was recently also proved by Ishii and Nimura for arbitrary complex reflection groups of rank 2 and real reflection groups of rank 3, but our approach is very different.
\end{abstract}

\section{Introduction}

Bounded derived categories of coherent sheaves are an interesting invariant of varieties which controls some of the more classical invariants such as cohomology or the Grothendieck group.

Often, when studying derived categories, one tries to lift interesting structures from the more classical invariants to the level of derived categories. This can take the form of categorification of group or Lie algebra actions. But also, somewhat more elementary, whenever there is a natural direct sum decomposition of cohomology or the Grothendieck group, one can ask whether this lifts to a semi-orthogonal decomposition of the derived category.

Mainly inspired by the McKay correspondence, also derived categories of equivariant coherent sheaves are widely studied. If a finite group $G$ acts on a smooth complex variety $X$, this equivariant derived category is denoted by $\D_G(X)$ and is canonically identified with the derived category $\D([X/G])$ of the associated quotient orbifold. There is always a natural direct sum decomposition of orbifold cohomology, sometimes used as the very definition of orbifold cohomology, namely
\begin{equation}\label{eq:orbidec}
 \Ho^*_{\mathsf{orb}}([X/G])\cong \bigoplus_{[g]\in \conj(G)}\Ho^*(X^g, \IQ)^{\CC(g)}
\end{equation}
where the direct sum runs over a set of representatives of the conjugacy classes of the group, $X^g$ denotes the locus of points fixed by $g\in G$, and $\CC(g)$ is the centraliser, whose action on $X$ restricts to $X^g$.

One cannot hope that the direct sum decomposition \eqref{eq:orbidec} always lifts to the derived category. Indeed, there are quotient orbifolds where it is easily seen via equivariant Serre duality that $\D_G(X)$ does not have any semi-orthogonal decomposition at all.
However, there is the following promising conjecture predicting when there should be such a lift.
\begin{conjecture}[\cite{PvdB}]\label{conj:PvdB}
Let a finite group $G$ act faithfully on a smooth variety $X$ over a field $\IK$ of characteristic zero in such a way that all the quotients $X^{(g)}:=X^g/\CC(g)$ are smooth. Then, there exists a semi-orthogonal decomposition
\begin{equation}\label{eq:conjsod}
 \D_G(X)=\bigl\langle \cA_g\mid [g]\in \conj(G) \bigr\rangle
\end{equation}
whose pieces are in bijection to the conjugacy classes of $G$ and satisfy $\cA_g\cong \D(X^{(g)})$.
\end{conjecture}
 In the same paper \cite{PvdB} where Polishchuk and van den Bergh stated the conjecture, they also proved it in interesting special cases, maybe most notably for the action of the symmetric group on the product  of a smooth curve by permutation of factors.

 If $G\subset \GL(V)$ acts linearly on a complex vector space $V$, the condition that $V/G$ is smooth is equivalent to $G$ being a complex reflection group by the Chevalley--Shephard--Todd Theorem \cite{ST--reflection}, \cite{Chev--reflections}. As, by Cartan's Lemma \cite[Lem.\ 2]{Cartan-Quot}, every action on a complex manifold is locally linearisable, complex reflection groups in some sense are the local case of the set-up of \autoref{conj:PvdB}.
%

The complex reflection groups are completely classified by \cite{ST--reflection}. There is the infinite series $G(m,e,k)\le \GL(\IC^k)$ with three integer parameters where $e\mid m$. Groups in this series are the \emph{primitive} complex reflection groups. In addition, there are 34 sporadic cases, usually denoted $G_4,\dots, G_{37}$ and called \emph{imprimitive} reflection groups. In \cite{PvdB}, \autoref{conj:PvdB} is proved in the case of imprimitive reflection groups with $e=1$, where $G(m,1,k)$ is the wreath product of $\mu_m$ by $\sym_n$, and for some Weyl groups.

There is further work in rank $2$. Namely, \cite{Potter} proved \autoref{conj:PvdB} for the dihedral groups $D(m,m,2)$; see also the related work \cite{Capellan}. In \cite{Lim-Rota}, the conjecture was proven for the group $G(4,2,2)$. Recently, \cite{Faberetal} proved \autoref{conj:PvdB} in the case of $G(2e,e,2)$ for arbitrary $e$ and for
$G_{12}$, $G_{13}$, $G_{22}$, which are exactly the three imprimitive groups of rank 2 with the property that all reflections have only the real Eigenvalues $\pm1$.

A few days before the writing of the present paper was finished, the preprint \cite{Ishii-Nimura} was uploaded on the arXiv. In it, \autoref{conj:PvdB} is deduced for arbitrary rank $2$ complex reflection groups from results of \cite{Kaw--toricIII}, and  \autoref{conj:PvdB} is also proved for all rank 3 real reflection groups.

The main result of our paper is
\begin{theorem}\label{thm:main}
\autoref{conj:PvdB} is true for every imprimitive rank two reflection group $G(m,e,2)$ with its given action on $V=\IK^2$.
\end{theorem}

Our approach has a very different flavour than that of \cite{Faberetal}, \cite{Ishii-Nimura}, and \cite{Kaw--toricIII}. These papers heavily use the McKay correspondence, while we stay on the equivariant side of this correspondence the whole time. There is hope that our more direct approach has a better chance to be generalised to imprimitive reflection groups of higher rank. Also, the present paper contains information which might not be so easy to extract from \cite{Kaw--toricIII}, like an explicit description of the exceptional sequence part of \eqref{eq:conjsod} in terms of the irreducible representations of $G(m,e,2)$. In summary, the author thinks his construction is still worth being presented, even though \autoref{thm:main} is now superseded by \cite[Cor.\ 3.2]{Ishii-Nimura}.


For the concrete construction of the pieces of our decomposition, we have to refer to \autoref{sect:sod}. However, to give a quick indication that the constructions are somewhat natural, let us note that the objects contained in the admissible subcategory corresponding to the conjugation class of some $g\in G$ are all supported on $Z_g:=G\cdot V^g$.

Also note that our semi-orthogonal decomposition is linear over the derived category $\D(V/G)$; see \autoref{subsect:linear}. This possible additional property of the conjectured semi-orthogonal decomposition was already mentioned in the introduction of \cite{PvdB}, and later sometimes even stated as part of the conjecture; see \cite[Def.\ 2.2]{Lim--abelian}. The $\D(V/G)$-linearity should be useful to deduce \autoref{conj:PvdB} in its global from the local case of complex reflection groups using decent of semi-orthogonal decompositions \cite{BS--descent}.

\medskip
\noindent
\textbf{Organisation of the Paper and Structure of the Proof.}

In \autoref{sect:equi} we do some preparations concerning equivariant coherent sheaves and their derived categories. In \autoref{subsect:equibasic}, we fix the general notation, and prove a few lemmas which are simple, and probably well-known, but were we could not find references. In \autoref{subsect:pifibre} we study the scheme-theoretic fibres of the quotient morphism $\pi\colon V\to V/G$ for $G\le \GL(V)$ a finite reflection group. In \autoref{subsect:reflectionff}, we construct the pieces of \eqref{eq:conjsod} corresponding to a given reflection $g\in G$ and its powers, under the condition that the canonical morphism $X^{(g)}\to Z_g/G$ is an isomorphism. This condition is not always fulfilled. However, sometimes it is, so we formulated \autoref{subsect:reflectionff} in greater generality than necessary for our set-up, in the hope that this will safe some time when checking other cases of \autoref{conj:PvdB}.

In \autoref{sect:G}, we study the representation theory of $G(m,e,2)$, and in \autoref{sect:sod}, we finally construct a semi-orthogonal decomposition of the form \eqref{eq:conjsod} for all imprimitive rank 2 reflection groups $G(m,e,2)$, which proves \autoref{thm:main}. In both of these sections, we often need to make a case distinction along the parity of $e$. The two main differences are:
\begin{itemize}
 \item For $e$ even, the centre of $G(m,e,2)$ is bigger and, related to this, the group has more one-dimensional representations.
 \item For $e$ even, the condition that the canonical morphism $V^{(g)}\to Z_g/G$ is an isomorphism is fulfilled for all reflections. In contrast, for $e$ odd, there is a reflection $\tau$ where this does not hold; see \autoref{subsect:tauodd}. So we cannot apply the general results of \autoref{subsect:reflectionff} and need to make some extra effort to construct the associated piece of \eqref{eq:conjsod}.
\end{itemize}

In \autoref{subsect:linear}, we quickly discuss the fact that our decomposition is $\D(V/G)$-linear. In \autoref{subsect:groundfield}, we generalise \autoref{thm:main} to a not algebraically closed ground field.

In the Appendix, we provide the McKay quiver of $G(m,e,2)$, again requiring a distinction between $e$ even and $e$ odd. The computation necessary to get the quivers is already performed in \autoref{subsect:skyExt}. Hence, it would have made sense to display the quivers already there, also because they are heavily used in \autoref{sect:sod}. The reason that we moved them to an appendix is that they only fit on a page in landscape format, and this format probably looks less confusing at the end of a paper.

\medskip
\noindent
\textbf{Conventions.}
We work over an algebraically closed field $\IK$ of characteristic zero. The assumption that the field is algebraically closed might or might not be necessary for some arguments of ours proofs. Anyway, in the final \autoref{subsect:groundfield}, we explain that for every field of characteristic zero over which $G(m,e,2)$ is defined, \autoref{thm:main} can be deduced from the statement for its algebraic closure by faithfully flat descent.

In this paper, the (equivariant) derived category $\D(X)$ (or $\D_G(X)$) of some scheme (with $G$-action) always means the bounded derived category of coherent sheaves.

\medskip
\noindent
\textbf{Acknowledgements.}
The author thanks Erik Nikolov for some helpful discussions.

\section{Generalities on Equivariant Sheaves}\label{sect:equi}

\subsection{Preliminaries on Equivariant Sheaves}\label{subsect:equibasic}
In this subsection, we only sketch a few simple facts about equivariant coherent sheaves and their derived categories and functors that we will need later. More detailed general references on this topic are \cite{BOber--equi}, \cite[Sect.\ 2.2]{Krug--remarksMcKay}, \cite[Sect.\ 4]{BKR}.

Let a finite group $G$ act on a scheme $X$. A \emph{$G$-equivariant sheaf} is a pair $(E,\lambda)$ where $E\in \Coh(X)$ and $\lambda=\{\lambda_g\colon E\xrightarrow \cong g^*E\}_{g\in G}$ is a family of isomorphisms, called $G$-\emph{linearisation}, such that the following diagrams for $g,h\in G$ commute
\[
\begin{tikzcd}
E \arrow[rrr, bend right=18, "\lambda_{hg}" ]      \arrow[r, , "\lambda_{g}"]         & g^*E    \arrow[r, "g^*\lambda_{h}"]           & g^*h^*F \arrow[r, "\cong"] & (hg)^*F \,.  \\
\end{tikzcd}
\]
We will often omit $\lambda$ in the notation and simply write $E$ for the equivariant sheaf.
For two $G$-equivariant sheaves $E$ and $F$, there is an induced $G$-action, given by conjugation by the $G$-linearisations, on $\Hom(E,F)$.
The equivariant sheaves on $V$ form an abelian category $\Coh_G(X)$ with homomorphisms
\[
 \Hom_G(E,F):=\Hom(E,F)^G\,.
\]
We denote the derived functors of $\Hom_G$ by $\Ext^i_G$.
We write $\D_G(X)$ for the bounded derived category of $\Coh_G(X)$. For $E,F\in \Coh_G(X)$, we have
\[
 \Hom_{\D_G(X)}(E, F[i])\cong \Ext^i_G(E,F)\,.
\]
Given $E\in \Coh_G(X)$, the linearisation of $E$ induces a $G$-action on the global sections $\Gamma(X,E)$.
\begin{lemma}\label{lem:invares}
 Let $H\le G$ be a subgroup, let $Y\subset X$ be an $H$-invariant reduced subscheme, and let $Z:=G\cdot Y\subset X$ equipped with the reduced subscheme structure. Then, for any $E\in \Coh_G(Z)$, the restriction map on invariant global sections
 \[
  \Gamma(Z, E)^G\to \Gamma(Y, E_{\mid Y})^H
 \]
is injective. In particular, $\Gamma(Y, E_{\mid Y})^H=0$ implies $\Gamma(Z, E)^G=0$.
\end{lemma}

\begin{proof}
 Let $s\in \Gamma(Z, E)^G$ with $s_{\mid Y}=0$. By $G$-invariance of $s$, we get $s_{\mid g\cdot Y}=0$ for every $g\in G$. As, by definition of $Z=G\cdot Y$, the $g\cdot Y$ cover $Z$, we have $s_{\mid Z}=0$.
\end{proof}

Let $H\le G$ be a subgroup. There is the \emph{restriction functor} $\Res_G^H\colon \Coh_G(X)\to \Coh_H(X)$ given by $\Res_G^H(E,\lambda)=(E,\lambda_{\mid H})$.
It has the \emph{induction functor} $\Ind\colon\Coh_H(X)\to \Coh_G(X)$ as a both-sided adjoint. It is given by $\Ind_H^G(E)=\bigoplus g^*E$ where the direct sum runs through a set of representatives of the co-sets $G/H$ and the linearisation is given by a combination of the $H$-linearisation of $E$ and permutation of the direct summands; see e.g.\ \cite[Sect.\ 3.2]{BOber--equi} for details. We have the useful formula
\begin{equation}\label{eq:Indprojformula}
 \Ind_H^G(E)\otimes F\cong \Ind_H^G\bigl( E\otimes \Res_H^G F \bigr)\quad\text{ for }E\in \Coh_H(X)\,,\, F\in \Coh_G(X)\,.
\end{equation}
In particular, if $\chi$ is a $G$-character whose restriction to $H$ is trivial, we have
\begin{equation}\label{eq:Indinva}
\Ind_H^G(E)\otimes \chi\cong \Ind_H^G(E)\,.
\end{equation}

\begin{lemma}\label{lem:Indcriterion}
Let $(E,\lambda)\in \Coh_G(X)$ such that there is a direct sum decomposition $E\cong \oplus_{i\in I} E_i$ of the underlying (non-equivariant)
sheaf together with a group action of $G$ on $I$ such that $\lambda_g(E_i)=g^*(E_{g(i)})$ for all $i\in I$. Fixing some $i_0\in I$, and denoting the stabiliser of $i_0\in I$ by $G_{i_0}$, we have $E\cong \Ind_{G_{i_0}}^G E_i$.
\end{lemma}

\begin{proof}
Under the assumptions, a $G$-equivariant homomorphism $E\to F$ to some $F\in \Coh_G(X)$ is determined by its restriction to the direct summand $E_{i_0}$, which is $G_{i_0}$-equivariant. Hence, $E$ satisfies the universal property of $\Ind_{G_{i_0}}^G E_i$.
\end{proof}

If $G$ acts trivially on $X$, a $G$-linearisation of a sheaf is the same as a $G$-action. In this case, we have another pair of both-sided adjoint functors, namely
\[
 (\_)^G\colon \Coh_G(X)\to \Coh(X)\quad, \quad \triv\colon \Coh(X)\to \Coh_G(X)\,.
\]
The functor $(\_)^G$ takes invariants of an equivariant sheaf under its given $G$-action, while $\triv$ equips a sheaf with the trivial $G$-action.

Let $G$ act on two schemes $X$ and $Y$, and let $f\colon X\to Y$ be a $G$-equivariant morphism. There is an equivariant pull-back $f^*\colon \Coh_G(Y)\to \Coh_G(X)$ and, if $f$ is proper, also an equivariant push-forward $f_*\colon \Coh_G(X)\to \Coh_G(Y)$. Both functors are compatible with the usual pull-back and push-forwards of the underlying non-equivariant sheaves.

\begin{lemma}\label{lem:flatfibrerep}
 Let $f\colon X\to Y$ be a flat and finite $G$-invariant morphism, and for $y\in Y$ let $X_y:=f^{-1}(y)$ denote the scheme-theoretic fibre. The $G$-representations $\Gamma(\reg_{X_y})$ are isomorphic for all $y\in Y$.
\end{lemma}

\begin{proof}
Let $W\in \irr(G)$ with dual $W^\vee$. By flatness, $f_*(\reg_X\otimes W^\vee)$ is a locally free sheaf on $Y$. Hence, its direct summand $f_*(\reg_X\otimes W^\vee)^G$ is locally free too. The fibres of this locally free sheaf are
\[
 f_*(\reg_X\otimes W^\vee)^G(y)\cong \Hom_G\bigl(W,\Gamma(\reg_{X_y}))\,.
\]
Hence, the multiplicity of $W$ as a direct summand of $\Gamma(\reg_{X_y})$ is independent of $y$.
\end{proof}

\begin{lemma}\label{lem:equifilt}
Let $x\in X$ with stabiliser $G_x\le G$, and $E\in \Coh_{G_x}(X)$ with $\supp E=\{x\}$ and
\[
 \Gamma(E)\cong \bigoplus_{W\in \irr(G_x)} W^{\oplus m(W)}\,.
\]
Then, there is a filtration of $E$ by $G_x$-equivariant subsheaves whose graded pieces are $\reg_x\otimes W$ occurring with multiplicity $m(W)$.
\end{lemma}

\begin{proof}
Let $\iota_x\colon \{x\}\hookrightarrow V$ be the embedding of the point. By $G_x$-equivariance, we have $\iota_{x*}\iota_x^* E\cong \reg_x\otimes U$ for some $G_x$-representation $U$ which is a quotient of $\Gamma(E)$. We pick some irreducible summand $W$ of $U$ and get a $G_x$-equivariant surjection $E\twoheadrightarrow \reg_x\otimes W$. We proceed by induction on $\dim \Gamma(E)$ to get a filtration of $\ker(E\twoheadrightarrow \reg_x\otimes W)$.
\end{proof}

\subsection{Fibres of Quotients by Finite Reflection Groups}\label{subsect:pifibre}

Let a finite group $G\le \GL(V)$ act on a finite-dimensional vector space $V$. We call an element $g\in G$ a \emph{reflection}\footnote{This is often called a \emph{pseudo-reflection} in the literature, while the term reflection then is reserved for pseudo-reflections of order $2$. For us, reflections can be of arbitrary order.} if its fixed point locus $V^g\subset V$ is a divisor in $V$.
We call $G$ a \emph{finite reflection group} if it is generated by reflections. By the Chevalley--Sheppard--Todd Theorem, this is equivalent to the quotient $V/G$ being smooth.

We denote the regular representation of $G$ by $\IK\langle G\rangle \cong \Ind_1^G \IK$.

\begin{lemma}\label{lem:pifibres}
Let $V$ be a vector space and $G\le \Aut(V)$ a finite reflection group. Let $v\in V$, and denote by $\eta=\pi^{-1}(\pi(v))$ the scheme theoretic fibre of the quotient morphism $\pi\colon V\to V/G$ over $\pi(v)$. Then $\reg_{\eta}\cong \Ind_{G_v}^G\reg_{\eta,v}$ where $\reg_{\eta,v}$ is the stalk, and
$\Gamma(\reg_{\eta,v})\cong \IK\langle G_v\rangle$.

Furthermore, $\reg_{\eta}$ has a filtration by $G$-equivariant subsheaves whose graded pieces are $\Ind_{G_v}^G(\reg_v\otimes W)$ with $W\in \irr(G_v)$ occurring with multiplicity $\dim W$.
\end{lemma}

\begin{proof}
The isomorphism $\reg_{\eta}\cong \Ind_{G_v}^G\reg_{\eta,v}$ follows directly by \autoref{lem:Indcriterion}.
For a general point $v\in V$, we have $G_v=1$. Then $\eta=G\cdot v$ is the free orbit with the reduced subscheme structure, hence $\reg_\eta=\Ind_1^G\reg_v$. This implies $\Gamma(\reg_\eta)\cong \IK\langle G\rangle$. Hence, by \autoref{lem:flatfibrerep}, we get
$\Gamma(\reg_\eta)\cong \IK\langle G\rangle$ for \emph{every} $v\in V$.

In particular, if $v$ is a $G$-fixed point, in other words $G_v=G$, we have $\reg_\eta\cong \reg_{\eta,v}$ with $\Gamma(\reg_{\eta},v)\cong \IK\langle G\rangle= \IK\langle G_v\rangle$ as asserted.\footnote{For $v=0$, this is \cite[Thm.\ (B)]{Chev--inva}, which we will use later in \autoref{subsect:irrep} to determine the irreducible representations of $G(m,e,2)$.}

We deduce the general case from the case that $G=G_v$ as follows using the diagram
\[
\begin{tikzcd}
V \arrow[dr, "\pi" ]      \arrow[r, , "p"]         & V/G_v    \arrow[d, "q"]  \\          & V/G
\end{tikzcd}
\]
where $p$ is the $G_v$-quotient morphism. By Steinberg's theorem \cite[Thm.\ 1.5]{Steinberg--refl}, \cite[Prop.\ 4.7]{Broue--book}, the stabiliser $G_v$ is again generated by reflections of $V$. Hence, we can replace $G$ by $G_v$ and use that we already confirmed the case of a point fixed under the whole reflection group to get $\Gamma(\reg_{p^{-1}(p(v))})\cong \IK\langle G_v\rangle$ where $p^{-1}(p(v))$ denotes the scheme-theoretic fibre. We  have $p^{-1}(p(v))\subset \eta$, hence
\begin{equation}\label{eq:hosurj}
\Gamma(\reg_{\eta,v})\twoheadrightarrow \Gamma(\reg_{p^{-1}(p(v))})\cong \IK\langle G_v\rangle\,.
\end{equation}
We already know that $\IK\langle G\rangle \cong \Gamma(\reg_{\eta})=\Ind_{G_v}^G \Gamma(\reg_{\eta,v})$, hence $\dim \Gamma(\reg_{\eta,v})=|G_v|=\dim \IK\langle G_v\rangle$. Thus, the surjection \eqref{eq:hosurj} is an isomorphism.

The last part of the statement is just \autoref{lem:equifilt} together with the decomposition of the regular representation $\IK\langle G_v\rangle$ into irreducibles.
\end{proof}

\subsection{Fully Faithful Functors Associated to Reflections}\label{subsect:reflectionff}

In this subsection, we assume throughout that the variety $X$ on which the finite group $G$ acts is smooth, and that the action is faithful.
We will give some results on fully faithful embeddings and semi-orthogonal decompositions associated to reflections. As general references for semi-orthogonal decompositions, we refer to \cite[Sect.\ 2.2]{Kuz--HPD}, \cite{Kuz--ICM}.

As mentioned in the introduction of \cite{PvdB}, the following simple fact was one of the motivations for \autoref{conj:PvdB}. It ensures that at least one of the pieces of the conjectured semi-orthogonal decomposition is always there.
\begin{lemma}\label{lem:piff}
If the quotient $X/G$ is smooth, then
$\pi^*\colon \D(X/G)\to \D_G(X)$ is fully faithful and its image is an admissible subcategory of $\D_G(X)$.
\end{lemma}
Note that $\pi^*\colon \D(X/G)\to \D_G(X)$ is a shortened notation for the composition
\[
 \D(X/G)\xrightarrow{\triv}\D_G(X/G)\xrightarrow{\pi^*} \D_G(X)\,.
\]
\begin{proof}
 This is essentially the projection formula together with the fact that for any geometric quotient $(\pi_*\reg_X)^G\cong \reg_{X/G}$.
\end{proof}

For $g\in G$, we write $X^g\subset X$ for the fixed point locus and $X^{(g)}:=X^g/\CC(g)$ for the quotient by the centraliser. We define
\[
Z_g:=G\cdot X^g=\bigcup_{h\in G} h\cdot X^g=\bigcup_{[h]\in G/\CC(g)} h\cdot X^g\subset X
\]
and consider it as a reduced subscheme of $X$. Note that the $G$-action on $X$ restricts to a $G$-action on $Z_g$.
We denote by
\[
 X^g\xhookrightarrow{\alpha_g} Z_g\xhookrightarrow{\iota_g} X\quad,\quad \pi_g\colon Z_g\to Z_g/G
\]
the closed embeddings and the quotient morphism.
The composition
\[
X^g \xhookrightarrow{\alpha_g} Z_g\xrightarrow{\pi_g} Z_g/G
\]
is $\CC(g)$-invariant, so it induces a morphism $\nu_g\colon X^{(g)}\to Z_g/G$.
For the rest of this subsection, we will make
\begin{assumption}\label{ass:g}
 Let $X^g\subset X$ be a connected divisor, $X^{(g)}$ be smooth, and the canonical morphism $\nu_g\colon X^{(g)}\to Z_g/G$ be an isomorphism.
\end{assumption}

Note that, under the assumption, the normal bundle $N_{X^g/X}\cong \reg_{X^g}(X^g)$ is a $\CC(g)$-equivariant line bundle. As $\langle g\rangle\le \CC(g)$ acts trivially on $X^g$, the $\CC(g)$-linearisation of $\reg_{X^g}(X^g)$ restricts to a $\langle g\rangle$-action. Since the $G$-action on $X$ is faithful, the action on $\reg_{X^g}(X^g)$ is multiplication by a generator of the group of $\langle g\rangle$-characters, which we denote by $\chi_g$.

Note that, generically (away from the intersection of the $h\cdot X^g$), the two line bundles $\reg_{X^g}(X^g)$ and $\reg_{Z_g}(Z_g)_{\mid X^g}$ on $X^g$ coincide. Hence $\langle g\rangle$ also acts on $\reg_{Z_g}(Z_g)_{\mid X^g}$ by multiplication by $\chi_g$.

\begin{prop}\label{prop:reflectionff}
Under \autoref{ass:g}, the morphism $\pi_g\colon Z_g\to Z_g/G$ is flat and the exact functor
 \begin{equation}\label{eq:Phig}
\Phi_g\colon  \D(X^{(g)})\xrightarrow{\nu_{g*}}\D(Z_g/G)\xrightarrow{\triv}\D_G(Z_g/G)\xrightarrow{\pi_g^*}\D_G(Z_g)\xrightarrow{\iota_{g*}}\D_G(X)
 \end{equation}
has a left and a right adjoint, and is fully faithful.
\end{prop}

\begin{proof}
By assumption $Z_g/G$ is smooth, and $Z_g$ is a divisor in the smooth $X$, hence Cohen--Macaulay. Thus, by miracle flatness, $\pi_g$ is flat; see e.g.\ \cite[Ex.\ 18.17]{Eisenbud--commutativebook}.

The maps $\pi_g$ and $\iota_g$ are proper, so all adjoints of their pull-backs and push-forwards exist by Grothendieck duality on the level of derived categories with appropriate boundedness conditions. As $X$ is smooth and $\pi_g$ is finite, the adjoints preserve the bounded derived categories of coherent sheaves that we are working with.

Since $\nu_{g*}$ is an equivalence, we can omit it from the composition when checking fully faithfulness. So, for the sake of this proof, we redefine $\Phi_g:=\iota_{g*}\circ \pi_g^*\circ \triv$. When plugging objects into this functor, we also omit $\triv$ from the notation, tacitly assuming that $\pi_g^*(E)$ is the $G$-equivariant object with the canonical linearisation coming from the pull-back. We will show that $\Phi_g^R\circ \Phi_g\cong \id$ where
\[
 \Phi_g^R\colon \D_G(X)\xrightarrow{\iota_{g}^!} \D_G(Z_g)\xrightarrow{\pi_{g*}} \D_G(Z_g/G)\xrightarrow{(\_)^G}\D(Z_g/G)
\]
is the right-adjoint of $\Phi_g$.
Note that $\iota_g^!\cong \iota_g(\_)^*\otimes \reg_{Z_g}(Z_g)$.
We write $\pi_{g*}^G:=(\_)^G\circ \pi_{g*}$. With this notation, for $E\in \D(Z_g/G)$, we have
$\Phi_g^R\circ \Phi_g\cong \pi_{g*}^G\bigl(\iota_g^!\iota_{g*}\pi_g^*(E)\bigr)$. In $\D_G(Z_g)$, there is the exact triangle
\begin{equation}\label{eq:adjtriangle}
 \pi_g^*(E)\to \iota_g^!\iota_{g*}\pi_g^*(E)\to \pi_g^*(E)\otimes \reg_{Z_g}(Z_g)[-1]\to
\end{equation}
We apply $\pi_{g*}^G$ to this triangle. We first prove that the right term of the triangle vanishes after applying $\pi_{g*}^G$. By projection formula
$\pi_{g*}^G\bigl(\pi_g^*(E)\otimes \reg_{Z_g}(Z_g)\bigr)\cong E\otimes \pi_{g*}^G\bigl(\reg_{Z_g}(Z_g)\bigr)$.
Recall the discussion preceding \autoref{prop:reflectionff} saying that $\langle g\rangle$ acts trivially on $X^g$ but non-trivially on the line bundle $\reg_{Z_g}(Z_g)_{\mid X^g}$. Hence, for every open subset $U\subset Z_g/G$, we have
$\Gamma\bigl(\pi_g^{-1}(U), \reg_{Z_g}(Z_g)_{\mid X^g}\bigr)^{\langle g\rangle}= 0$.
By \autoref{lem:invares}, this implies $\Gamma\bigl(\pi_g^{-1}(U), \reg_{Z_g}(Z_g)\bigr)^{G}= 0$, which in turn implies $\pi_{g*}^G \reg_{Z_g}(Z_g)=0$.

Hence, the first map of \eqref{eq:adjtriangle} becomes an isomorphism after applying $\pi_{g*}$.
As $\pi_g$ is a quotient, we have $\pi_{g*}^G\reg_{Z_g}\cong \reg_{Z_g/G}$. Hence, by projection formula,
$\pi_{g*}^G\pi_g^*(E)\cong E\otimes \pi_{g*}^G\reg_{Z_g}\cong E$.

As applying $\pi_{g*}^G$ to the middle term of \eqref{eq:adjtriangle} gives $\Phi_g^R\Phi_g(E)$, we get $E\cong \Phi_g^R\Phi_g(E)$. This isomorphism is functorial in $E$ because the first map of \eqref{eq:adjtriangle} is the unit of adjunction.
\end{proof}

For a $G$-character $\mu$, we write $\Phi_g^\mu:=\Phi_g(\_)\otimes \mu$. Since tensor product by $\mu$ is an autoequivalence, $\Phi_g^\mu$
is again fully faithful.

\begin{lemma}\label{lem:Phiorth}
 Let \autoref{ass:g} hold, and let $\mu_1,\mu_2$ be $G$-characters such that \[\Res_{G}^{\langle g\rangle}(\mu_1^\vee\otimes \mu_2)\not\cong \mathbf 1_g \quad,\quad \Res_{G}^{\langle g\rangle}(\mu_1^\vee\otimes \mu_2)\not\cong \mathbf \chi_g^\vee\] where $\mathbf 1_g$ is the trivial $\langle g\rangle$-character, and $\chi_g$ the generating character by which $\langle g\rangle$ acts on $\reg_{X^g}(X^g)$ and $\reg_{X_g}(X_g)_{\mid X^g}$. Then $(\Phi_g^{\mu_1})^R\circ \Phi_g^{\mu_2}\cong 0$.
\end{lemma}

\begin{proof}
 We have $(\Phi_g^{\mu_1})^R \cong \Phi_g^R(\_\otimes \mu_1^\vee)$. Hence, in analogy to the proof of \autoref{prop:reflectionff}, we get an exact triangle
\begin{equation}\label{eq:adjtriangle2}
 \pi_g^*(E)\otimes \mu_1^\vee\otimes \mu_2 \to \iota_g^!\iota_{g*}\pi_g^*(E)\otimes \mu_1^\vee\otimes \mu_2\to \pi_g^*(E)\otimes \reg_{Z_g}(Z_g)\otimes\mu_1^\vee\otimes \mu_2[-1]\to
\end{equation}
such that applying $\pi_{g*}^G$ to the middle term gives $(\Phi_g^{\mu_1})^R \Phi_g^{\mu_2}(E)$. By the assumptions on the characters, $\langle g\rangle$ acts trivially on the restrictions of $\reg_{Z_g}\otimes \mu_1^\vee\otimes \mu_2$ and $\reg_{Z_g}(Z_g)\otimes \mu_1^\vee\otimes \mu_2$ to $X^g$. Again in analogy to the proof of \autoref{prop:reflectionff}, this gives that the outer terms of
\eqref{eq:adjtriangle2} vanish after applying $\pi_{g*}^G$. Hence, the same holds for the middle term.
\end{proof}

For a $G$-character $\mu$, we write $\cA(\mu):=\Phi_g^{\mu}\bigl(\D(X^{(g)})\bigr)=\Phi_g\bigl(\D(X^{(g)})\bigr)\otimes \mu$.

\begin{cor}\label{cor:reflectionsod}
Let \autoref{ass:g} hold, and let $\mu$ be a $G$-character whose restriction to $\langle g\rangle$ is $\chi_g^\vee$. Then there exists a semi-orthogonal decomposition
\begin{equation}\label{eq:Asodgeneral}
 \bigl\langle \cA(\mu), \cA(\mu^2),\dots, \cA(\mu^{\ord(g)-1})\bigr\rangle
\end{equation}
of some admissible subcategory of $\D_G(X)$. If $X/G$ is smooth, then the admissible subcategory $\pi^*\D(X/G)$ is left-orthogonal to \eqref{eq:Asodgeneral}. Hence, we can extend \eqref{eq:Asodgeneral} to the semi-orthogonal decomposition
\begin{equation}\label{eq:Asodgeneral+}
 \bigl\langle  \cA(\mu), \cA(\mu^2),\dots, \cA(\mu^{\ord(g)-1}), \pi^*\D(X/G) \bigr\rangle
\end{equation}
of some admissible subcategory of $\D_G(X)$.
\end{cor}

\begin{proof}
The decomposition \eqref{eq:Asodgeneral} follows directly by \autoref{prop:reflectionff} and \autoref{lem:Phiorth}.
For \eqref{eq:Asodgeneral+}, due to \autoref{lem:piff}, we only need to check that $\pi^*\D(X/G)$ is left-orthogonal to the $\cA(\mu^c)$ which is equivalent to $\pi_*^G\circ \Phi_g^{\mu^c}\cong 0$ for $c=1,\dots, \ord(g)$.

The vanishing of the exact functor can be tested by its values on sheaves $E\in \Coh(Z_g/G)$. We have $\pi_*^G \Phi_g^{\mu^c}(E)\cong \pi_*^G(\iota_{g*}\pi_g^*(E)\otimes \mu^c)$. The vanishing of this follows from the fact that $\langle g\rangle$ acts non-trivially on $(\pi_g^*(E)\otimes \mu^c)_{\mid X^g}$ and \autoref{lem:invares}.
\end{proof}

\begin{remark}\label{rem:muexists}
A $G$-character $\mu$ as in \autoref{cor:reflectionsod} always exists. As $Z_g$ is a $G$-invariant effective divisor, there exists some $G$-equivariant line bundle $L$ and $f\in \Gamma(X,L)$ such that $Z_g$ is the vanishing locus of $G$. The section $f$ is semi-invariant under the $G$-action on $\Gamma(X,L)$, and we can take the associated character $\mu=\chi(f):=\langle f\rangle\subset \Gamma(X,L)$.
\end{remark}

\begin{remark}
\autoref{cor:reflectionsod} provides, under \autoref{ass:g}, the pieces of the semi-orthogonal decomposition of \autoref{conj:PvdB} associated to all the powers of a given reflection. Note, however, that the condition that $\nu_g\colon X^{(g)}\to Z_g$ is an isomorphism is not as mild as one might hope. In fact, we will see  in \autoref{subsect:tauodd} that for the reflection group $G(m,e,2)$ with $e$ odd, there is a reflection not satisfying this condition. The weaker condition that $\nu_g$ is birational is studied in \cite[Sect.\ 2]{PvdB}.
\end{remark}

\section{Representation Theory of Rank Two Imprimitive Reflection Groups}\label{sect:G}

In this section, we study the imprimitive rank 2 reflection group $G(m,e,2)$, its irreducible representations, and the extension groups between skyscraper sheaves in the origin $0\in \IK^2$ twisted by irreducible presentations, which is related to the McKay graphs displayed in \autoref{app:even} and \autoref{app:odd}.

Most of the results presented in this section are well-known, but the notation and the presentation in the litarature differs much from ours, which is why we decided to keep this section largely self-contained.
For a description of the irreducible representations of $G(m,e,k)$; see \cite[Sect.\ 6]{Stembridge--eigenvalues}. A more detailed study of the $k=2$ case, including the McKay graphs of $G(m,1,2)$, is provided in \cite{May--Gmp2}.

\subsection{Description of the Group Elements}
Let us fix the two-dimensional vector space $V:= \IK^2$.
Let $d,e$ be positive integers, and set $m:=de$. Let $\mu_m\subset \IK^*$ denote the group of $m$-th roots of unity. We define the abelian subgroup $H:=H(m,e,2)\le \GL(2,\IK)=\GL(V)$ by
\[
H(m,e,2)=\left\{\begin{pmatrix}
                  \eta& 0\\ 0 & \theta
                 \end{pmatrix}
\middle\vert \eta, \theta\in \mu_m\,,\, (\eta\theta)^d=1
 \right\}\,.
\]
We fix once and for all a primitive $m$-th root of unity $\zeta$, i.e.\ a generator of $\mu_m$,
and write
\[
 D(a,b)=\begin{pmatrix}
                  \zeta^a& 0\\ 0 & \zeta ^b
                 \end{pmatrix} \quad\text{for }a,b\in \IZ/m\IZ\,.
\]
Using this notation, we can rewrite $H$ as
\begin{equation}\label{eq:Hdefzeta}
 H(m,e,2)=\bigl\{D(a,b)\mid a,b\in \IZ/m\IZ\,,\, a+b\equiv 0 \mod e
 \bigr\}
\end{equation}
We write $\gamma:=\zeta^e$, which is a primitive $d$-th root of unity. The group $H$ is generated by the two elements
\[
\tilde \zeta:=\begin{pmatrix}
                  \zeta& 0\\ 0 & \zeta^{-1}
                 \end{pmatrix} \quad,\quad
\gamma_y:=\begin{pmatrix}
                  1& 0\\ 0 & \gamma
                 \end{pmatrix}\,.
\]
More precisely, we have an isomorphism
$\mu_m\times \mu_d\xrightarrow\cong H$ given by $(\zeta^a, \gamma^c)\mapsto \tilde\zeta^a\gamma_y^c$.
Setting
\[
 \gamma_x:=\begin{pmatrix}
                  \gamma& 0\\ 0 & 1
                 \end{pmatrix}
\]
we also have $\tilde \zeta$, $\gamma_x$ as another pair of generators.
We denote the permutation matrix by
\[
 \tau=\begin{pmatrix}
                  0& 1\\ 1 & 0
                 \end{pmatrix}\,.
\]
Note that we have $D(a,b)\tau=\tau D(b,a)$. In other words,
$D(a,b)^\tau=D(b,a)$.
The finite reflection group $G:=G(m,e,2)$ is defined as the subgroup of $\GL(V)$ generated by $H$ and $\tau$:
\[
G=H\rtimes \sym_2=H\cup H\tau\,.
\]
\subsection{Conjugacy Classes}
As $H\lhd G$ is an abelian subgroup of index 2, the $G$-conjugacy classes of elements in $H$ consist of at most two elements. More precisely, for $a\neq b$, the two elements $D(a,b)$ and $D(b,a)$ are conjugated. Elements of the form $D(a,a)\in H$, i.e.\ multiples of the unit matrix, are only conjugated to themselves, hence form the center $\Z(G)$ of the group.

For the matrix $D(a,a)$ to actually be an element of $G$, we must have $2a\equiv 0\mod e$; see \eqref{eq:Hdefzeta}. This leads to an important difference between the case that $e$ is odd, where $\Z(G)=\langle D(e,e)\rangle\cong \mu_d$, and the case that $e$ is even, where $\Z(G)=\langle D(\frac e2,\frac e2)\rangle\cong \mu_{2d}$.

We now look at the $G$-conjugacy classes of elements in $H\tau$. We have
\begin{equation}\label{eq:conjHtau}
 (D(a,b)\tau)^{\tilde \zeta}=D(a-2, b+2)\tau\quad, \quad (D(a,b)\tau)^{\gamma_y}=D(a-e,b-e)\,.
\end{equation}
In the case that $e$ is \emph{odd}, a complete list of conjugacy classes in $H\tau$ is given by the following $d$ representatives:
\begin{equation}\label{eq:conjrepodd}
 \tau\,,\,\, \gamma_y\tau=D(0,e)\tau \,,\,\, \gamma_y^2\tau=D(0,2e)\tau \,,\,\, \dots \,,\,\, \gamma_y^{d-1}\tau=D(0,(d-1)e)\tau\,.
\end{equation}
Indeed, $\det(\gamma_y^c\tau)=-\gamma^c$, so no two different matrices in the list \eqref{eq:conjrepodd} can be conjugated. Furthermore,
as $e$ and $2$ are coprime, \eqref{eq:conjHtau} tells us that every $D(a,b)\tau$ is conjugated to an element of the form $D(0,*)\tau$, which is then contained in the list \eqref{eq:conjrepodd}.

If $e$ is \emph{even}, the condition $a+b\equiv 0\mod e$ needed for $D(a,b)\in H$ implies that $a$ and $b$ must have the same parity. Furthermore, by \eqref{eq:conjHtau}, the parity of $a$ and $b$ in an element $D(a,b)\tau$ is preserved under conjugation. This yields a complete list of $G$-conjugacy classes in $H\tau$ with the following $2d$ representatives:
\begin{equation}\label{eq:conjrepeven}
\begin{split}
 &\tau\,,\,\, \gamma_y\tau=D(0,e)\tau \,,\,\, \dots \,,\,\, \gamma_y^{d-1}\tau=D(0,(d-1)e)\tau\,,\\
 &\tilde\zeta\tau=D(1,-1)\tau\,,\,\, \tilde\zeta\gamma_y\tau=D(1,e-1)\tau  \,,\,\, \dots \,,\,\, \tilde\zeta\gamma_y^{d-1}\tau=D(1,(d-1)e-1)\tau\,.
\end{split}
 \end{equation}

\subsection{Fixed Point Loci}\label{subsect:fixed}

Recall from \autoref{subsect:pifibre} that we call a $g\in G\le \GL(V)$ whose fixed point locus is a hyperplane simply a \emph{reflection} (we do not use the term \emph{pseudoreflection}, even if $\ord(g)\neq 2$). Of course, since in our set-up $V=\IK^2$, this means that the fixed point locus $V^g$ of a reflection is a line.

There are $2(d-1)$ reflections in $H$, namely $\gamma_x\,,\, \gamma_y \,,\, \gamma_x^2\,,\, \gamma_y^2, \dots, \gamma_x^{d-1}\,,\, \gamma_y^{d-1}$.
The fixed point loci, independently of $c=1,\dots, d-1$, are
\[
 V^{\gamma_x^c}=\left\langle\binom01\right\rangle = \{0\}\times \IA^1 =\Van(x)\quad, \quad V^{\gamma_y^c}=\left\langle\binom10\right\rangle =\IA^1\times \{0\} =\Van(y)
\]
where for $f\in \IK[x,y]$, we denote its vanishing locus by $\Van(f)\subset V=\IA^2$.
For every $c=1,\dots, d-1$, the two reflections $\gamma_x^c$ and $\gamma_y^c$ form a conjugated pair. The centralisers are
\[\CC(\gamma_x^c)=\CC(\gamma_y^c)=H=\langle \tzeta\rangle\times \langle \gamma_x\rangle= \langle \tzeta\rangle\times \langle \gamma_y\rangle\,.\]
On $V^{\gamma_y^c}\cong \IA^1$, the generator $\gamma_y$ acts trivially, while $\tilde\zeta$ acts as multiplication by $\zeta$.
Hence,
\begin{equation}\label{eq:reginvagamma}
V^{(\gamma_y^c)}\cong\IA^1/\mu_m\cong \Spec\IK[x^m]\cong \IA^1\,.
\end{equation}
There are $m$ reflections in $H\tau$, namely
$\tau\,,\, \tilde\zeta\tau\,,\, \tilde\zeta^2\tau\,,\,\dots \,,\, \tilde\zeta^{m-1}\tau$.
Their fixed point loci are
\[
 V^{\tilde\zeta^i\tau}=\left\langle\binom{\zeta^i}1\right\rangle =\Van(x-\zeta^i y)
\]
Setting $t_i:=[x]=[\zeta^iy]$, we get $\reg(V^{\tilde\zeta^i\tau})\cong \IK[t_i]$.

If $e$ is \emph{odd}, all the $\tilde\zeta^i\tau$ for $i=0,\dots, m-1$ form just one conjugation class. We usually represent this class by $\tau$ which has centraliser
$\CC(\tau)=\langle D(e,e)\rangle\times \langle \tau\rangle$.
The second factor $\langle\tau\rangle$ acts trivially on the diagonal $V^\tau$, while $D(e,e)$ acts by multiplication by the primitive $d$-th root of unity $\gamma=\zeta^e$. Hence,
\begin{equation}\label{eq:tauoddinva}
 V^{(\tau)}\cong \Spec \IK[t_0^d]\cong \IA^1  \qquad\text{(e odd)}\,.
\end{equation}
For $e$ \emph{even}, the set $\{\tilde\zeta^i\tau\mid i=0,\dots, m-1\}$ decomposes into two conjugacy classes. Namely, into those $\tilde\zeta^i\tau$ with $i$ even, represented by $\tau$, and those with $i$ odd, represented by $\tilde\zeta \tau$. We now have $\CC(\tau)=\langle D(\frac e2,\frac e2)\rangle\times \langle \tau\rangle$ and $\CC(\tzeta \tau)=\langle D(\frac e2,\frac e2)\rangle\times \langle \tzeta\tau\rangle$ which leads to
\begin{equation}\label{eq:reginvataueven}
 V^{(\tau)}\cong \Spec \IK[t_0^{2d}]\cong \IA^1\quad, \quad V^{(\tzeta\tau)}\cong \Spec \IK[t_1^{2d}]\cong \IA^1\qquad\text{(e even).}
\end{equation}
Let us summarize the discussion of the previous two subsections in two lemmas.
\begin{lemma}\label{lem:conjodd}
 Let $e$ be odd. Then the conjugacy classes of $G$ are:
 \begin{itemize}
  \item $\frac{(m-1)d}2$ conjugated pairs $D(a,b)\sim D(b,a)$,
\item $d$ central elements $D(\lambda e,\lambda e)$,
\item $d=\frac {md}m$ conjugacy classes in $H\tau$, each consisting of $m$ elements.
  \end{itemize}
Hence, there are $\frac{(m-1)d}2+2d$ conjugacy classes in $G$. Among them are $d$ conjugacy classes of reflections represented by
$\gamma_y\,,\,  \gamma_y^2\,,\,\dots \,,\, \gamma_y^{d-1}\,,\, \tau$.
\end{lemma}

\begin{lemma}\label{lem:conjeven}
 Let $e$ be even. Then the conjugacy classes of $G$ are:
 \begin{itemize}
  \item $\frac{(m-2)d}2$ conjugated pairs $D(a,b)\sim D(b,a)$,
\item $2d$ central elements $D(\lambda \frac e2,\lambda \frac e2)$,
\item $2d=\frac {md}{m/2}$ conjugacy classes in $H\tau$, each consisting of $\frac m2$ elements.
  \end{itemize}
Hence, there are $\frac{(m-2)d}2+4d$ conjugacy classes in $G$. Among them are $d+1$ conjugacy classes of reflections represented by
$\gamma_y\,,\,  \gamma_y^2\,,\,\dots \,,\, \gamma_y^{d-1}\,,\, \tau \,,\, \tzeta\tau$.
\end{lemma}

\subsection{Irreducible Representations - Generalities}\label{subsect:irrep}

In the following, we will describe the set $\irr(G)$ of isomorphism classes of irreducible $G$-representations.

We use the induced action of $G$ on $\reg(V)=\IK[x,y]$. The invariants are freely generated by the two polynomials $(xy)^d$ and $x^m+y^m$; see e.g.\ \cite[Rem.\ 2.5(ii)]{Cohen--refl}. The quotient by the ideal $I:=((xy)^d, x^m+y^m)$ generated by the invariants is isomorphic to the regular $G$-representation
\begin{equation}
\IK[x,y]/I\cong \IK\langle G\rangle\,;
\end{equation}
see \cite[Thm.\ (B)]{Chev--inva}, or \autoref{lem:pifibres}, noting that $I$ is the vanishing ideal of the scheme-theoretic fibre $\pi^{-1}(\pi(0))$ of the quotient morphism. A $\IK$-base of $\IK[x,y]/I$ is given by
\[
 \cB=\bigl\{x^ay^b\mid \min\{a,b\}<d\,,\, \max\{a,b\}<m \bigr\}\cup \bigl\{Ax^ay^b\mid a,b<d  \bigr\}\quad,\quad A:=x^m-y^m\,.
\]
Probably the quickest way to check that $\cB$ is a base is to note that it has the correct number of elements $|\cB|=2md=|G|=\dim \IK[x,y]/I$ and then confirm that it is a generating system using the relations $A+(x^m+y^m)=2x^m$ and $A-(x^m+y^m)=-2y^m$.

Among the polynomials in the base $\cB$, there are $2d$ which are $G$-semi-invariant, namely
\[
 1\,,\, xy \,,\, (xy)^2 \,,\, \dots\,,\, (xy)^{d-1}\,,\, A\,,\, Axy \,,\, A(xy)^2 \,,\, \dots\,,\, A(xy)^{d-1}\,.
\]
For a $G$-semi-invariant polynomial $f$, we write
\[
 \chi(f)=\langle f\rangle
\]
which is a one-dimensional, hence irreducible, $G$-representation.
The other polynomials in $\cB$ form pairs $x^ay^b, x^by^a$ and $Ax^ay^b,Ax^by^a$ permuted by $\tau$. These pairs span 2-dimensional $G$-representations, which we denote by
\[
 \rho(x^ay^b):=\langle x^ay^b, x^by^a \rangle\quad ,\quad \rho(Ax^ay^b):=\langle Ax^ay^b, Ax^by^a \rangle \quad \text{with $a>b$.}
\]
The choice that we label the representation by the generator $x^ay^b$ with $a>b$ and not by the other generator of the same $G$-invariant subspace is  arbitrary, but some choice has to be made if we want a unique label.

One can summarise our notation for the $\chi$'s and $\rho$'s by noting that, if we label a $G$-representation by a polynomial, we mean the $G$-invariant subspace of $\IC[x,y]$ spanned by the polynomial's $G$-orbit.

The monomials $x^ay^b$ are $H$-semi-invariant and we denote the corresponding one-dimensional $H$-representations by $\chi_H(x^ay^b)$. With this notation, we have
\begin{equation}\label{eq:rhoind}
\rho(x^ay^b)\cong \Ind_H^G \chi_H(x^ay^b)\,.
\end{equation}
Hence, we can decide which of the $\rho(x^ay^b)$ are irreducible by
Mackey's Irreducibility Criterion (see e.g.\ \cite[Sect.\ 7.4]{Serre--repbook}), which in our special case reads
\begin{lemma}\label{lem:Mackey}
 The 2-dimensional $G$-representation $\rho(x^ay^b)\cong \Ind_H^G \chi_H(x^ay^b)$ is irreducible if and only if
 \[
  \chi_H(x^ay^b)^\tau\not\cong \chi_H(x^ay^b)
 \]
where the $H$-representation $\chi_H(x^ay^b)^\tau$ is given by letting $h\in H$ act on $\chi_H(x^ay^b)$ by $h^\tau$.
\end{lemma}

Note that $H$ acts trivially on $A=x^m-y^m$, which gives
\begin{equation}\label{eq:rhoA}
 \chi_H(Ax^ay^b)\cong \chi_H(x^ay^b)\quad\text{and} \quad  \rho(Ax^ay^b)\cong \rho(x^ay^b)\,.
\end{equation}
This means that we can ignore the $\rho(Ax^ay^b)$ when we classify the isomorphism classes of irreducible representations.
Hence, it suffices to formulate the following lemma only for $\chi_H(x^ay^b)$ and not for $\chi_H(Ax^ay^b)$.

\begin{lemma}\label{lem:conjcondition}
 We have $\chi_H(x^ay^b)^\tau\cong \chi_H(x^ay^b)$ if and only if
 \begin{equation}\label{eq:conjcondition}
2(a-b)\equiv 0 \mod m\quad\text{and}\quad a\equiv b \mod d\,.
 \end{equation}
\end{lemma}

\begin{proof}
The $H$-character $\chi_H(x^ay^b)$ is given by
 \begin{equation}\label{eq:chiH}
 \chi_H(x^ay^b)\colon H\to \IK^*\,,\quad \tzeta\mapsto \zeta^{a-b}\,,\, \gamma_x\mapsto \gamma^a \,,\, \gamma_y\mapsto \gamma^b\,.
 \end{equation}
We have $\tzeta^\tau=\tzeta^{-1}$, $\gamma_x^\tau=\gamma_y$, and $\gamma_y^\tau=\gamma_x$. Hence
 \begin{equation}\label{eq:chiHconj}
 \chi_H(x^ay^b)^\tau\colon H\to \IK^*\,,\quad \tzeta\mapsto \zeta^{b-a}\,,\, \gamma_x\mapsto \gamma^b \,,\, \gamma_y\mapsto \gamma^a\,.
 \end{equation}
The result follows by comparing \eqref{eq:chiH} and \eqref{eq:chiHconj}.
\end{proof}

\begin{lemma}\label{lem:camparerhos}
 Let $\lambda\in \{1,\dots, e-1\}$ and $a,b\in \{0,\dots, d-1\}$. Then
\[
\rho(x^{\lambda d+a}y^b)\cong \rho(x^{(e-\lambda)d+b}y^a)\,.
\]
\end{lemma}

\begin{proof}
 The computation is similar to the proof of \autoref{lem:conjcondition}. We have
\[
 \chi_H(x^{\lambda d+a}y^b)\colon \quad \tzeta\mapsto \zeta^{\lambda d+a-b}\,,\, \gamma_x\mapsto \gamma^{\lambda d+a}=\gamma^a \,,\, \gamma_y\mapsto \gamma^b\,.
 \]
Hence, its $\tau$-conjugation is given by
\[
 \chi_H(x^{\lambda d+a}y^b)^{\tau}\colon \quad \tzeta\mapsto \zeta^{-(\lambda d+a-b)}=\zeta^{(e-\lambda)d+b-a}\,,\, \gamma_x\mapsto \gamma^b \,,\, \gamma_y\mapsto \gamma^a
 \]
which agrees with $\chi_H(x^{(e-\lambda)d+b}y^a)$.
Thus, we have
\[
\rho(x^{\lambda d+a}y^b)\cong \Ind \chi_H(x^{\lambda d+a}y^b)\cong \Ind \chi_H(x^{\lambda d+a}y^b)^\tau\cong \Ind \chi_H(x^{(e-\lambda)d+b}y^a)\cong \rho(x^{(e-\lambda)d+b}y^a)\,.\qedhere
\]
\end{proof}

\subsection{Irreducible Representations - odd case}

In this subsection, we complete the description of $\irr(G)$ in the case that $e$ is odd.

\begin{lemma}[$e$ odd]
All the $G$-representations $\rho(x^ay^b)$ for $x^ay^b\in \cB$ with $a>b$ are irreducible.
\end{lemma}

\begin{proof}
By \autoref{lem:Mackey} and \autoref{lem:conjcondition}, we have to show that no $x^ay^b\in \cB$ with $a>b$ satisfies \eqref{eq:conjcondition}.
We need to make a case distinction by the parity of $m$ (which agrees with the parity of $d$ since $e$ is odd). For $m$ odd, the condition $2(a-b)\equiv 0 \mod m$ implies that $a\equiv b\mod m$. This is incompatible with the assumption $m>a>b\ge 0$.

 For $m$ even, the condition $2(a-b)\equiv 0 \mod m$ is indeed satisfied for one monomial $x^ay^b\in \cB$ with $a>b$, namely for $x^{m/2}=x^{m/2}y^0$. However, $\frac m2\equiv \frac d2\mod d$, so the second condition of \eqref{eq:conjcondition} is not met.
\end{proof}

Every 2-dimensional irreducible representation occurs twice as a direct summand of the regular representation $\IK\langle G\rangle\cong \IK[x,y]/I$.
So, to get a classification of the isomorphism classes, we have to find the isomorphic pairs.
As noted above, we have $\rho(x^ay^b)\cong \rho(Ax^ay^b)$ for $a,b<d$; see \eqref{eq:rhoA}. Set
\[
 n^-:=\frac{m-d}2=\frac{(e-1)}2 d\quad,\quad n^+=\frac{m+d}2 =\frac{(e+1)}2d\,.
\]
\autoref{lem:camparerhos} then says that for every $x^ay^b\in \cB$ with $b<d\le a<n^+$ there is some $x^\alpha y^\beta\in \cB$ with $\beta<d$, $\alpha\ge n^+$ and
$\rho(x^ay^b)\cong \rho(x^\alpha y^\beta)$.
Let us summarise the discussion as

\begin{prop}[$e$ odd] The group $G$ has $2d$ characters, namely
\begin{align*}
 \mathbf 1=\chi(1), \chi(xy), \dots \chi((xy)^{d-1})\quad &, \quad \chi(A), \chi(Axy),\dots,\chi(A(xy)^{d-1}) \,.
\end{align*}
In other words, the group of characters is generated by $\chi(xy)$, which is of order $d$, and $\chi(A)$, which is of order $2$.
Furthermore, $G$ has $\frac{(m-1)d}2$ isomorphism classes of $2$-dimensional irreducible representations. A set of representatives of these isomorphism classes is given by
\[
\rho(x^ay^b) \quad\text{with }  b<d,\, b< a< n^+\,.
\]
There are no irreducible $G$-representations of dimension greater than 2.
\end{prop}

The reader might already want to have a look at diagram \eqref{eq:McKayodd}. The meaning of the arrows is explained later by \autoref{prop:skyExt}, but the nodes of the displayed quiver are almost in bijection with $\irr(G)$ (why only almost is explained in the text before the diagram).

\subsection{Irreducible Representations - even case}\label{subsect:irreven}

We now consider the case that $e$ is even. Of course, this implies that $m=de$ is even, and we set
\[n:=\frac m2=\frac e2 d\,.\]

Now, in contrast to the case that $e$ is odd, there are $d$ monomials $x^ay^b\in \cB$ with $a>b$ satisfying \eqref{eq:conjcondition}, namely
$x^n(xy)^b=x^{n+b}y^b$ for $b=0,\dots, d-1$.

Hence, \autoref{lem:Mackey} and \autoref{lem:conjcondition} tell us that $\rho(x^n(xy)^b)=\langle x^n(xy)^b, y^n(xy)^b\rangle$ splits into two one-dimensional $G$-representations. Indeed, there are two $G$-semi-invariant polynomials
\begin{equation}\label{eq:Ndef}
 N^+:=x^n+y^n\,,\, N^-:=x^n-y^n\in \rho(x^n)\,.
\end{equation}
This gives $\rho(x^n(xy)^b)= \chi(N^+(xy)^b)\oplus \chi(N^-(xy)^b)$. Note that
\begin{equation}\label{eq:chiNArel}
\chi(N^+)\otimes\chi(N^+)\cong \chi(1)\cong \chi(N^-)\otimes \chi(N^-)\quad,\quad \chi(N^-)\cong \chi(N^+)\otimes\chi(A)\quad,\quad
\end{equation}
Let now $x^ay^b\in \cB$ with $a>b$ be not of the form $x^n(xy)^b$, in other words, $a\neq n+b$. Then one checks that $x^ay^b$ does not satisfy \eqref{eq:conjcondition}, so the corresponding $2$-dimensional representation $\rho(x^ay^b)$ is irreducible.

To give a set of representatives of the isomorphism classes of 2-dimensional irreducible representations, we have to identify the isomorphic pairs in $\IK[x,y]/I$. Note that we still have $\rho(x^ay^b)\cong \rho(Ax^ay^b)$ for $b<a<d$; see \eqref{eq:rhoA}. Furthermore, note that for $b<d\le a$, we have
\begin{equation}\label{eq:evenrhocompare}
 \rho\bigl(x^a(xy)^b\bigr)\cong \rho\bigl(x^{m-a}(xy)^{[a+b]}\bigr)
\end{equation}
where $[a+b]$ denotes the unique number in $\{0,\dots ,d-1\}$ with $[a+b]\equiv a+b\mod d$. This can either be deduced from the statement of \autoref{lem:camparerhos} or, maybe simpler, verified by doing a computation with the corresponding $H$-characters as in the proof of \autoref{lem:camparerhos}.

Noting that $a<n$ if and only if $m-a\ge n$, we get that every isomorphism class of 2-dimensional $G$-representations has a unique representative $\rho\bigl(x^{a}(xy)^b\bigr)$ with $0<a<n$ and $0\le b<d$.
We summarise our discussion as
\begin{prop}[$e$ even]\label{prop:irreven} The group $G$ has $4d$ characters, namely
\begin{align*}
 \chi(1), \chi(xy), \dots \chi((xy)^{d-1})\quad &, \quad \chi(A), \chi(Axy),\dots,\chi(A(xy)^{d-1}) \\
 \chi(N^+), \chi(N^+xy), \dots \chi(N^+(xy)^{d-1})\quad &, \quad \chi(N^-), \chi(N^-xy),\dots,\chi(N^-(xy)^{d-1})\,.
\end{align*}
In other words, the group of characters is generated by $\chi(xy)$, which is of order $d$, and $\chi(A)$, $\chi(N^+)$, which are of order $2$.
Furthermore, $G$ has $\frac{(m-2)d}2=(n-1)d$ isomorphism classes of $2$-dimensional irreducible representations. A set of representatives of these isomorphism classes is given by
\[
\rho\bigl(x^{a}(xy)^b\bigr) \quad\text{with }  0<a<n\,,\, 0\le b<d\,.
\]
There are no irreducible $G$-representations of dimension greater than 2.
\end{prop}
The reader is invited to take a look at diagram \eqref{eq:evenMcKay} whose nodes correspond to $\irr(G)$.

\subsection{$G$-equivariant extension groups}

Let us collect some facts on equivariant sheaves in $\Coh_G(V)\subset \D_G(V)$ and their extension groups in our set-up of $G=G(m,e,2)$ acting on $V=\IK^2$.

As $V=\Spec \IK[x,y]$ is affine, an equivariant sheaf can be identified with an $A$-module $M$ together with a $G$ action on $M$ such that $g(am)=g(a)g(m)$. This, in turn, can be identified with a module over the skew group algebra $G\ltimes A$ (see e.g.\ \cite[Sect.\ 2.5]{KruNik}), but we will not use this description.

Note that for $E\in \Coh_G(V)$ and $W\in \rep(G)$, there is a well-defined $E\otimes W\in \Coh_G(V)$. For $\chi(f)$ the character given by some semi-invariant $f\in \IK[x,y]$, we abbreviate $E\otimes \chi(f)=:E(f)$.

As we consider $\IK[x,y]=\reg(V)$ with the action induced by the one on $V$, we have the isomorphism of $G$-representations
\[
 \rho(x)=\langle x,y\rangle \cong V^\vee\,.
\]
Hence, we have a natural isomorphism $\Omega_V\cong \reg_V\otimes \rho(x)$. Since $\det\rho(x)=\chi(Axy)$, the canonical bundle with its natural $G$-linearisation is given by $\omega_V\cong \reg_V(Axy)$. As $V$ is not proper, we do not have a global Serre duality for all equivariant sheaves on $V$. However, Serre duality still holds if at least one of the two sheaves involved has proper support; see \cite[Sect.\ 4.3]{BKR}. As our $V=\IK^2$ is affine, having proper support is the same as having zero-dimensional support. Concretely, this means
\begin{equation}\label{eq:Serre}
 \Ext^*_G(E,F)\cong \Ext^*_G\bigl(F, E(Axy)\bigr)^\vee[-2]\quad\text{for } E,F\in \Coh_G(V) \text{ with } \dim E=0\text{ or } \dim F =0\,.
\end{equation}
Since $V$ is affine, we can compute $\Ext^i_G$ for arbitrary $E,F\in \Coh_G(V)$ as
\begin{equation}\label{eq:Extres}
 \Ext^i_G(E,F)\cong \Gamma\Bigl(V, \cH^i\bigl(\Hom(K^\bullet,F)  \bigr) \Bigr)^G
\end{equation}
where $K^\bullet$ is an equivariant locally free resolution of $E$.

\subsection{Extension Groups of Equivariant Skyscraper Sheaves}\label{subsect:skyExt}

In this subsection, we compute the extension groups between skyscraper sheaves of the origin twisted by irreducible representations.

\begin{lemma}\label{lem:skyExt}
Let $U,W$ be irreducible $G$-representations. Then
 \begin{align*}
        \Hom_G(\reg_0\otimes U, \reg_0\otimes W)&\cong \Hom_G(U,W)\cong \begin{cases}
                                                       \IK \quad\text{if $U\cong W$,}\\
                                                      0 \quad\text{else.}
                                                      \end{cases}
\\
        \Ext^1_G(\reg_0\otimes U, \reg_0\otimes W)&\cong \Hom_G\bigl(U\otimes \rho(x), W\bigr)
\\
  \Ext^2_G(\reg_0\otimes U, \reg_0\otimes W)&\cong \Hom_G\bigl(U\otimes \chi(Axy),W\bigr)\cong \begin{cases}
                                                       \IK \quad\text{if $U\otimes \chi(Axy)\cong W$,}\\
                                                      0 \quad\text{else.}
                                                      \end{cases}
       \end{align*}
The second formula can be rephrased as $\Ext^1_G(\reg_0\otimes U, \reg_0\otimes W)\cong\IK^r$, where $r$ is the multiplicity of $W$ as a direct summand of $U\otimes \rho(x)$.
       \end{lemma}

\begin{proof}
 The standard Koszul resolution of $\reg_0$ reads $G$-equivariantly
\[
0\to \reg_V(Axy)\xrightarrow{\binom{y}{-x}} \reg_V\otimes \rho(x)\xrightarrow{(x,y)}\reg_V\to \reg_0\to 0\,.
\]
Applying $\Hom(\_\otimes U,\reg_0\otimes W)$ gives the complex with vanishing differentials
\[
\reg_0\otimes \Hom(U,V)\xrightarrow{ \,\,0\,\,} \reg_0\otimes \Hom\bigl(U\otimes \rho(x),V\bigr) \xrightarrow{ \,\,0\,\,} \reg_0\otimes \Hom\bigl(U\otimes \chi(Axy),V\bigr)\,.
\]
Taking global sections and then invariants by \eqref{eq:Extres} gives the result.
\end{proof}

In the following lemma, which describes the tensor products of $G$-representations, for an arbitrary monomial $x^ay^b\in \IK[x,y]$, we interpret $\rho(x^ay^b)$ as $\rho(x^ay^b)\cong\Ind_H^G\chi_H(x^ay^b)$.

If $a\neq b$, this agrees with $\rho(x^ay^b)=\langle x^ay^b, x^by^a\rangle \subset \IK[x,y]$ as defined in \autoref{subsect:irrep}.
For $a=b$, however, we have $\rho(x^ay^a)\cong \chi((xy)^a)\oplus \chi(A(xy)^a)$.

\begin{lemma}\label{lem:tensor}
Let $x^ay^b\in \cB$ with $a>b$, and $c\in \{0,\dots, d-1\}$. Then
\begin{align}\label{eq:chirho}
 \rho(x^ay^b)\otimes \chi((xy)^c)&\cong \rho(x^ay^b)\otimes \chi(A(xy)^c)\cong  \rho(x^{a+c}y^{b+c})\,,\\
\label{eq:rhorho}
 \rho(x^ay^b)\otimes \rho(x)&\cong \rho(x^{a+1}y^{b})\oplus \rho(x^ay^{b+1}) \,.
\end{align}
In the case that $e$ is even, we also have
\begin{equation}\label{eq:chiNrho}
 \rho(x^ay^b)\otimes \chi(N^+(xy)^c)\cong \rho(x^{a}y^b)\otimes \chi(N^-(xy)^c)\cong  \rho(x^{a+n+c}y^{b+c})\,.
\end{equation}
\end{lemma}
Before proving \autoref{lem:tensor}, let us make some remarks.
Note that it can happen that $\rho(x^ay^b)$ is irreducible, but $\rho(x^{a+1}y^{b})$ or $\rho(x^ay^{b+1})$ is not. In this case,
the right side of \eqref{eq:rhorho} decomposes further. This is the case if $a=b+1$, where \eqref{eq:rhorho} gives
\begin{equation}\label{eq:xysplit}
 \rho(x^{b+1}y^{b})\otimes \rho(x)\cong \rho(x^{b+2}y^{b})\oplus \rho(x^{b+1}y^{b+1})\cong \rho(x^{b+2}y^{b})\oplus \chi(x^{b+1}y^{b+1})\oplus \chi(Ax^{b+1}y^{b+1})  \,.
\end{equation}
In the case that $e$ is even, this also happens for $a=n+c$, $b=c+1$ where \eqref{eq:rhorho} gives
\begin{equation}\label{eq:Nsplit}
\begin{split}
 \rho(x^{n+c}y^{c+1})\otimes \rho(x)&\cong \rho(x^{n+c+1}y^{c+1})\oplus \rho(x^{n+c}y^{c+2})\\&\cong \chi(N^+(xy)^{c+1})\oplus \chi(N^-(xy)^{c+1})\oplus \rho(x^{n+c}y^{c+2})  \,;
 \end{split}
\end{equation}
compare \eqref{eq:Ndef}.
One can easily generalise \eqref{eq:rhorho} to a formula for arbitrary $\rho(x^ay^b)\otimes \rho(x^{a'}y^{b'})$, but we will only need the special case
$\rho(x^ay^b)\otimes \rho(x)$ in the following.

\begin{proof}[Proof of \autoref{lem:tensor}]
For \eqref{eq:chirho}, we compute
\begin{align*}
 \rho(x^ay^b)\otimes \chi((xy)^c)\cong \bigl(\Ind_H^G \chi_H(x^ay^b) \bigr)\otimes \chi((xy)^c)&\cong \Ind_H^G\bigr(\chi_H(x^ay^b)\otimes\Res_G^H\chi((xy)^c)\bigr)\\& \cong
\Ind_H^G\bigr(\chi_H(x^ay^b)\otimes\chi_H((xy)^c)\bigr)\\
&\cong \Ind_H^G\bigr(\chi_H(x^{a+c}y^{b+c})\bigr)\\
&\cong \rho(x^{a+c}y^{b+c})\,.
\end{align*}
The proof for $\chi((xy)^c)$ replaced by $\chi(A(xy)^c)$ is the same, since
\[\Res \chi(A(xy)^c)\cong \Res \chi((xy)^c)\cong \chi_H((xy)^c)\,.\]
The proof of \eqref{eq:chiNrho} is still the same using that
\[\Res \chi(N^+(xy)^c)\cong \Res \chi(N^-(xy)^c)\cong \chi_H(x^{n+c}y^c)\,.\]
Also the proof of \eqref{eq:rhorho} is similar:
\begin{align*}
 \rho(x^ay^b)\otimes \rho(x)\cong \Ind\bigr(\chi_H(x^ay^b)\otimes\Res\rho(x)\bigr) & \cong \Ind\Bigr(\chi_H(x^ay^b)\otimes\bigl(\chi_H(x)\oplus \chi_H(y)\bigr)\Bigr)\\
 &\cong \bigl(\Ind\chi_H(x^{a+1}y^b)\bigr)\oplus\bigl(\Ind\chi_H(x^{a}y^{b+1})\bigr)\\
&\cong \rho(x^{a+1}y^{b})\oplus \rho(x^ay^{b+1})\,.\qedhere
\end{align*}
\end{proof}

\begin{prop}\label{prop:skyExt}
For all $U,W\in \irr(G)$ and $i=1,2$, we have
\[\dim\Ext^i_G(\reg_0\otimes U, \reg_0\otimes W)\le 1\,.\]
In other words, $\Ext^i_G(\reg_0\otimes U, \reg_0\otimes W)$ is either vanishing or one-dimensional.
The only pairs, up to isomorphism, $(U,W)$ such that $\Ext^1_G(\reg_0\otimes U, \reg_0\otimes W)\cong \IK$ are:
\begin{itemize}
 \item $(\chi((xy)^c), \rho(x^{c+1}y^c)\bigr)$, $(\chi(A(xy)^c), \rho(x^{c+1}y^c)\bigr)$, $\bigl(\rho(x^{c+1}y^c), \chi((xy)^{c+1})\bigr)$,\\ and $\bigl(\rho(x^{c+1}y^c), \chi(A(xy)^{c+1})\bigr)$ for $c=0,\dots, d-1$ (note that $\chi((xy)^d)\cong \chi(1)$),
 \item $\bigl(\rho(x^{a}y^b), \rho(x^{a}y^{b+1})\bigr)$ assuming that both, $\rho(x^{a}y^{b})$ and $\rho(x^{a}y^{b+1})$ are irreducible (note, for example, that $\rho(x^ay^a)$ splits into $\chi((xy^a))$ and $\chi(A(xy^a))$, so we are back in the case of the previous bullet point),
  \item $\bigl(\rho(x^{a}y^b), \rho(x^{a+1}y^b)\bigr)$ assuming that both, $\rho(x^{a}y^{b})$ and $\rho(x^{a+1}y^{b})$ are irreducible.
  \item In the case that $e$ is even, we have the additional pairs: $\bigl(\chi(N^+(xy)^c), \rho(x^{n+c}y^{c+1})\bigr)$, $\bigl(\chi(N^-(xy)^c), \rho(x^{n+c}y^{c+1})\bigr)$, $\bigl(\rho(x^{n+c}y^{c+1}),\chi(N^+(xy)^{c+1})\bigr)$,\\ $\bigl(\rho(x^{n+c}y^{c+1}),\chi(N^-(xy)^{c+1})\bigr)$ for $c=0,\dots, d-1$.
\end{itemize}
The only pairs, up to isomorphism, $(U,W)$ such that $\Ext^2_G(\reg_0\otimes U, \reg_0\otimes W)\cong \IK$ are:
\begin{itemize}
\item $\bigl(\chi((xy)^c), \chi(A(xy)^{c+1})\bigr)$ and $\bigl(\chi(A(xy)^c), \chi((xy)^{c+1})\bigr)$ for $c=0,\dots, d-1$,
\item $\bigl(\rho(x^{a}y^b), \rho(x^{a+1}y^{b+1})\bigr)$ if $\rho(x^ay^b)$ is irreducible (then, $\rho(x^{a+1}y^{b+1})\cong \rho(x^ay^b)\otimes \chi(Axy)$ is automatically irreducible too).
\item In the case that $e$ is even, we have the additional pairs: $\bigl(\chi(N^+(xy)^c), \chi(N^-(xy)^{c+1})\bigr)$ and $\bigl(\chi(N^-(xy)^c), \chi(N^+(xy)^{c+1})\bigr)$ for $c=0,\dots, d-1$.
\end{itemize}
\end{prop}
Note that the $\Hom_G(\reg_0\otimes U, \reg_0\otimes W)\cong \Ext^0_G(\reg_0\otimes U, \reg_0\otimes W)$ are already fully described by the first formula of \autoref{lem:skyExt}.

Stated like this, the content of the proposition might be a bit difficult to grasp. We give a diagrammatic depiction via McKay quivers in  \autoref{app:even} and \autoref{app:odd}, separating between the cases that $e$ is even and odd.

\begin{proof}
This follows from \autoref{lem:skyExt} and \autoref{lem:tensor} (together with its specifications \eqref{eq:xysplit} and \eqref{eq:Nsplit}) by going through all possible cases.

We will present the computation of $\Ext^i\bigl(\reg_0(N^\pm(xy)^c), \_\bigr)$, because in this case there is \autoref{lem:camparerhos} as a little extra ingredient, and leave the other cases to the reader.
By the third formula of \autoref{lem:skyExt} and \eqref{eq:chiNArel}, the only non-vanishing $\Ext^2_G\bigl(\reg_0(N^\pm(xy)^c), \reg_0\otimes W\bigr)$ is for \[W\cong \chi(N^\pm(xy)^c)\otimes \chi(Axy)\cong \chi(N^\mp(xy)^{c+1})\,,\] and it is one-dimensional, as asserted.
By \eqref{eq:chiNrho}, we have
\[
\rho(x)\otimes \chi(N^\pm (xy)^c)\cong \rho(x^{n+c+1}y^c)\cong \rho(x^{n+c}y^{c+1}) \,.
\]
where the second isomorphism is \autoref{lem:camparerhos} for $\lambda=\frac e2$, or \eqref{eq:evenrhocompare} with $x^{n+c+1}y^c=x^{n+1}(xy)^c$. Hence, by \autoref{lem:skyExt}, the only non-vanishing
$\Ext^1_G\bigl(\reg_0(N^\pm(xy)^c), \reg_0\otimes W\bigr)$ is
\[
 \Ext^1_G\bigl(\reg_0(N^\pm(xy)^c), \reg_0\otimes \rho(x^{n+c}y^{c+1})\bigr)\cong \IK\,.\qedhere
\]
\end{proof}

\begin{remark}
 All the $\Ext^1_G$-classes occurring in \autoref{prop:skyExt} can be explained by ideals inside $G$-invariant subschemes supported at the origin $0\in V$. For example, the non-vanishing class in $\Ext^1_G\bigl(\reg_0\otimes \rho(x^ay^b), \reg_0\otimes \rho(x^ay^{b+1})\bigr)$ corresponds to
 \[
  0\to \reg_0\otimes \rho(x^ay^{b+1})\to  \frac{(x^ay^b, x^by^a)}{(x^{a+1}y^b, x^by^{a+1})+\fm^{a+b+2}}  \to \reg_0\otimes \rho(x^ay^b)\to 0
 \]
where $\fm=(x,y)$ and the embedding of $\reg_0\otimes \rho(x^ay^{b+1})$ has image
$\frac{(x^ay^{b+1}, x^{b+1}y^a)}{\fm^{a+b+2}}$. A generalisation of the $b=0$ case of this are the sheaves $F(x^a)$ occurring later in \autoref{subsect:exceptionaleven}.

As another example, for $e$ even, we have $\Ext^1_G\bigl(\reg_0(N^+), \reg_0\otimes \rho(x^{n-1}y)\bigr)\cong \IK$. Using the isomorphism  $\rho(x^{n-1}y)\cong \rho(x^{n+1})$, this corresponds to
 \[
  0\to \reg_0\otimes \rho(x^{n+1})\to  \frac{(N^+, x^{n+1}, y^{n+1})}{(x^{n}y, xy^{n})+\fm^{n+2}}  \to \reg_0(N^+)\to 0\,.
 \]
\end{remark}

\begin{cor}\label{cor:skyexc}
The object $\reg_0\otimes W\in \D_G(V)$ is exceptional for every $W\in \irr(G)$.
\end{cor}

\begin{proof}
By \autoref{lem:skyExt}, we have $\Hom_G(\reg_0\otimes W, \reg_0\otimes W)\cong \IC$ for every $W\in \irr(G)$. Furthermore, no pair of the form $(W,W)$ occurs in the lists of non-vanish $\Ext^1_G$ and $\Ext^2_G$ in \autoref{prop:skyExt}.
\end{proof}

\section{The Semi-Orthogonal Decomposition}\label{sect:sod}

In this section we construct a semi-orthogonal decomposition of $\D_G(V)$ for $G=G(m,e,2)$ and $V=\IK^2$ with pieces as predicted by \autoref{conj:PvdB}. In other words, we prove \autoref{thm:main}.

In \autoref{subsect:gamma}, we provide the pieces whose construction is independent of the parity of $e$.
We then first construct the remaining pieces for $e$ even - those corresponding to reflections in \autoref{subsect:taueven}, and
the exceptional sequence corresponding to group elements which fix only the origin in
\autoref{subsect:exceptionaleven}. Afterwards, the same is done for $e$ odd in \autoref{subsect:tauodd} and \autoref{subsect:exceptionalodd}.

\subsection{The Components Associated to $\id$ and $\gamma_y$}\label{subsect:gamma}

We will now construct the pieces of the semi-orthogonal decomposition corresponding to the $d-1$ conjugacy classes represented by $\gamma_y^c$ for $c=1,\dots,d-1$, using the general results of \autoref{subsect:reflectionff}.

Recall that the $\gamma_y^c$ all have the same fixed point loci $V^{\gamma_y^c}=\IA^1\times \{0\}$, the same centralisers $\CC(\gamma_y^c)=H$, and hence the same associated quotients
\[V^{(\gamma_y^c)}:=V^{\gamma_y^c}/\CC(\gamma_y^c)\cong \IA^1\,.\]
 We will just write $Z_\gamma$ instead of $Z_{\gamma_y}$ for $G\cdot V^{\gamma_y}\subset V$ to ease the notation, and because $Z_{\gamma_y}=Z_{\gamma_x}$. Note that its vanishing ideal is
\begin{equation}\label{eq:IZgamma}
I(Z_\gamma)=(xy)\subset \IK[x,y]\,.
\end{equation}
The following shows that also the last part of \autoref{ass:g} is satisfied for $g=\gamma_y$.

\begin{lemma}\label{prop:gammaquot}
The canonical morphism $\nu_\gamma\colon V^{(\gamma_y)}\to Z_{\gamma}/G$ is an isomorphism.
\end{lemma}

\begin{proof}
The morphism $\nu_\gamma\colon V^{(\gamma_y)}\to Z_{\gamma}/G$ corresponds on regular functions to the restriction map
\begin{align*}
 \reg(Z_\gamma)^G\cong \left( \frac{\IK[x,y]}{A} \right)^G&\to \reg(V^{\gamma_y})^{\CC(\gamma_x)}\cong \IK[x]^{\mu_m}\cong \IK[x^m]\\
 f&\mapsto f_{\mid V^{\gamma_y}}\,.
\end{align*}
For the description of $\reg(V^{\gamma_y})^{\CC(\gamma_x)}$, see \eqref{eq:reginvagamma}. The restriction map is injective by \autoref{lem:invares}. It is surjective because the $G$-invariant function $x^m+y^m$ is sent to the generator $x^m$.
\end{proof}

\begin{prop}\label{prop:gammasod}
The functor
\[ \Phi_{\gamma}\colon\D(V^{(\gamma_y)})\xrightarrow{\nu_{\gamma*}} \D(Z_\gamma/G)\xrightarrow{\triv} \D_G(Z_\gamma/G)\xrightarrow{\pi_\gamma^*} \D_G(Z_{\gamma})\xrightarrow{\iota_{\gamma*}}\D_G(V)\]
is fully faithful. Denoting $\cA_\gamma:=\Phi_\gamma\bigl(\D(V^{(\gamma_y)})\bigr)$, there is a semi-orthogonal decomposition
\begin{equation}\label{eq:Asod}
\Bigl\langle \cA_\gamma(xy), \cA_\gamma((xy)^2)\dots, \cA_\gamma ((xy)^{d-1}), \pi^*(\D(V/G))   \Bigr\rangle
\end{equation}
of some admissible subcategory of $\D_G(V)$
\end{prop}

There occur $d-1$ copies of $\cA_\gamma\cong \D(\IA^1)$, namely all its images under tensor product by non-trivial powers of the character $\chi(xy)$.

\begin{proof}
The fully faithfulness of $\Phi_{\gamma}$ is \autoref{prop:reflectionff}.
By \autoref{rem:muexists} and \eqref{eq:IZgamma}, we can take $\mu=\chi(xy)$ in \autoref{cor:reflectionsod}, which gives the result.
\end{proof}

Let us proof two more lemmas about the the semi-orthogonal decomposition \eqref{eq:Asod} and the functor $\Phi_\gamma$ for later use.

\begin{lemma}\label{lem:complementA}
For every $W\in \irr(G)\setminus\{\chi((xy)^c)\mid c=0,\dots d-1\}$, the twisted skyscraper sheaf $\reg_0\otimes W$ is contained in
 $\bigl\langle \cA_\gamma(xy), \cA_\gamma((xy)^2)\dots, \cA_\gamma ((xy)^{d-1}), \pi^*(\D(V/G))   \bigl\rangle^\perp$, the right-orthogonal complement of \eqref{eq:Asod} in $\D_G(V)$.
\end{lemma}

\begin{proof}
Orthogonality can be checked on generators of the pieces of \eqref{eq:Asod}. Since the quotients are affine, generators of  $\cA_\gamma((xy)^c)$ and $\pi^*(\D(V/G))$ are
$\Phi_{\gamma}(\reg_{V^{(\gamma_y)}})((xy)^c)\cong \reg_{Z_\gamma}(A(xy)^c)$ and $\pi^*\reg_{V/G}\cong \reg_V$,
respectively.
Let $W\in \irr(G)$. We have
\[
\Ext^*_G(\reg_V, \reg_0\otimes W)\cong \Gamma(\reg_0\otimes W)^G[0]\cong  W^G[0]
\]
which vanishes except for $W\cong \chi(1)=\chi((xy)^0)$.
To the equivariant Koszul resolution
\begin{equation} \label{eq:gammaKoszul}
0\to \reg_V(xy)\xrightarrow{xy} \reg_V\to \reg_{Z_\gamma}\to 0\,,
\end{equation}
we apply $\Hom\bigl(\_\otimes \chi((xy)^c),\reg_0\otimes W\bigr)$. The resulting two-term complex has vanishing differential, and by \eqref{eq:Extres} we get
\[
\Ext^*_G\bigl(\reg_{Z_\gamma}((xy)^c), \reg_0\otimes W\bigr)\cong \Hom_G\bigl(\chi((xy)^c), W\bigr)[0]\oplus \Hom_G\bigl(\chi((xy)^{c-1}), W\bigr)[-1]\,.
\]
It follows that all $\Ext^*_G\bigl(\reg_{Z_\gamma}(A(xy)^c), \reg_0\otimes W\bigr)$ vanish as long as $W$ is not isomorphic to a character of the form $\chi(A(xy)^{c'})$ for some $c'$.
\end{proof}

\begin{lemma}\label{lem:Phigammaimage}
 The object $\Phi_\gamma(\reg_{\pi_\gamma(0)})$ is a sheaf (i.e.\ concentrated in degree zero) with a filtration whose graded pieces are
exactly the sheaves of the form $\reg_0\otimes W$ where $W\in \irr(G)$ with non-vanishing $\gamma_y$-invariants $W^{\gamma_y}\neq 0$.
\end{lemma}

Among the $G$-characters, the condition $W^{\gamma_y}\neq 0$ is fulfilled exactly for
\begin{itemize}
 \item the two characters $\chi(1)$ and $\chi(A)$ if $e$ is odd,
 \item the four characters $\chi(1)$, $\chi(A)$, $\chi(N^+)$, and $\chi(N^-)$ if $e$ is even.
\end{itemize}

\begin{proof}
Note that, by flatness of $\pi_\gamma$, for any $t\in V^{(\gamma_y)}\cong Z_\gamma/G$, we have $\Phi_\gamma(\reg_t)\cong \reg_{\pi_\gamma^{-1}(t)}$, considered as a sheaf on $V$, not on $Z_\tau$.
Let $s\in V^{\gamma_y}\setminus \{0\}$, which means that $G_s=\langle \gamma_y \rangle$, and let $t=\pi_\gamma(s)$. Then $\pi_\gamma^{-1}(t)$ is the reduced orbit $G\cdot s$. As, $\reg_{G\cdot s}=\Ind_{\langle \gamma_y\rangle}^G\reg_s$, we have $\Gamma(\reg_{\pi_\gamma^{-1}(t)})\cong \Ind_{\langle \gamma_y\rangle}^G \mathbf 1_\gamma$ where $\mathbf 1_\gamma$ denotes the trivial $\langle \gamma_y\rangle$-representation. By \autoref{lem:flatfibrerep}, also
\[
 \Gamma(\reg_{\pi_\gamma^{-1}(\pi_\gamma(0))})\cong \Ind_{\langle \gamma_y\rangle}^G \mathbf 1_\gamma\,.
\]
As $\reg_{\pi_\gamma^{-1}(\pi_\gamma(0))}$ is supported in the single point $0$, by \autoref{lem:equifilt}, it has a filtration whose graded pieces are exactly the $\reg_0\otimes W$ such that $W\in \irr(G)$ occurs in the decomposition of
$\Ind_{\langle \gamma_y\rangle}^G \mathbf 1_\gamma$ into irreducibles.
By adjunction,
\[
\Hom_G(\Ind_{\langle \gamma_y\rangle}^G \mathbf 1_\gamma, W) \cong \Hom_{\langle \gamma_y\rangle}(\mathbf 1_\gamma,\Res_G^{\langle \gamma_y\rangle} W)\cong W^{\gamma_y}\,.  \qedhere
\]
 \end{proof}

\begin{remark}
There is another way to get the same description of $\Gamma(\reg_{\pi_\gamma^{-1}(\pi_\gamma(0))})$ as in the proof, with the little bonus that we directly see the list of $W\in\irr(G)$ with $W^{\gamma_y}\neq 0$. Namely, one computes the vanishing ideal
$J:=I\bigl(\pi_\gamma^{-1}(\pi_\gamma(0))\bigr)=(xy, x^m+y^m)\subset \IK[x,y]$. Then,
$\IK[x,y]/J$ has the $\IC$-basis $1,x,y,\dots, x^{m-1}, y^{m-1}, A$. Hence, as $G$-representations,
\[
\Gamma(\reg_{\pi_\gamma^{-1}(\pi_\gamma(0))})=\IK[x,y]/J\cong \chi(1)\oplus \rho(x)\oplus\dots\oplus\rho(x^{m-1})\oplus \chi(A)\,.
\]
If $e$ is even, $\rho(x^n)$ in the middle of that sum decomposes further as $\rho(x^n)\cong \chi(N^+)\oplus \chi(N^-)$.
\end{remark}

\subsection{The Components Associated to $\tau$ - even case}\label{subsect:taueven}

In this and the next subsection, we assume that $e$ is even. In this case, we have two more conjugacy classes of reflections besides the $[\gamma_y^c]$ treated in the previous subsection. They are represented by $\tau$ and $\tzeta \tau$; see \autoref{lem:conjeven}. Using the notation analogous to \autoref{subsect:reflectionff} and \autoref{subsect:gamma}, we have $G$-invariant reduced subschemes $Z_\tau, Z_{\tzeta\tau}\subset V$ with vanishing ideals
\begin{equation}\label{eq:IZtau}
 I(Z_\tau)=(N^-)=(x^n-y^n)\quad,\quad I(Z_{\tzeta\tau})=(N^+)=(x^n+y^n)
\end{equation}

\begin{lemma}\label{prop:tauquot}
The canonical morphisms $\nu_\tau\colon V^{(\tau)}\to Z_\tau/G$ and $\nu_{\tzeta\tau}\colon V^{(\tilde\zeta \tau)}\to Z_{\tilde\zeta\tau}/G$
are isomorphisms.
\end{lemma}

\begin{proof}
As in the proof of \autoref{prop:gammaquot}, the induced map $V^{(\tau)}\to Z_\tau/G$ corresponds to
\begin{align*}
 \reg(Z_\tau)^G\cong \left( \frac{\IK[x,y]}{N^-} \right)^G&\to \reg(V^{\tau})^{\CC(\tau})\cong \IK[t_0]^{\mu_{2d}}\cong \IK[t_0^{2d}]\\
 f&\mapsto f_{\mid V^{\gamma_y}}
\end{align*}
where $t_0=[x]=[y]$; compare \eqref{eq:reginvataueven}. The restriction map is injective by \autoref{lem:invares}. It is surjective because the $G$-invariant function $(xy)^d$ is sent to the generator $t_0^{2d}$.

The proof for $V^{(\tilde\zeta \tau)}\to Z_{\tilde\zeta\tau}/G$ being an isomorphism is completely analogous.
\end{proof}

\begin{prop}\label{prop:Phitau}
The functors
\begin{align*}
&\Phi_\tau\colon \D(V^{(\tau)})\xrightarrow{\nu_{\tau*}}\D(Z_\tau/G)\xrightarrow{\triv} \D_G(Z_\tau/G)\xrightarrow{\pi_\tau^*} \D_G(Z_{\tau})\xrightarrow{\iota_{\tau*}}\D_G(V)\\
&\Phi_{\tilde \zeta\tau}\colon \D(V^{(\tzeta\tau)})\xrightarrow{\nu_{\tzeta\tau*}} \D(Z_{\tzeta\tau}/G)\xrightarrow{\triv}\D(Z_{\tzeta\tau}/G)\xrightarrow{\pi_{\tilde \zeta\tau^*}} \D_G(Z_{\tilde \zeta\tau})\xrightarrow{\iota_{\tilde \zeta\tau*}}\D_G(V)
\end{align*}
are both fully faithful.
\end{prop}

\begin{proof}
This is \autoref{prop:tauquot} together with \autoref{prop:reflectionff}.
\end{proof}

Before we can use these two fully faithful functors to extend the semi-orthogonal decomposition \eqref{eq:Asod}, we need another lemma.

\begin{lemma}\label{lem:caprep}
\begin{enumerate}
\item We have the following decompositions of $G$-representations \begin{align*}
\Gamma(\reg_{Z_\tau\cap Z_\gamma})&\cong \chi(1) \oplus \bigl(\bigoplus_{i=1}^{n-1} \rho(x^i)\bigr) \oplus \chi(N^+)\,,\\
\Gamma(\reg_{Z_{\tilde \zeta \tau}\cap Z_\gamma})&\cong \chi(1) \oplus \bigl(\bigoplus_{i=1}^{n-1} \rho(x^i)\bigr) \oplus \chi(N^-)\,.
\end{align*}
\item In the decomposition of the $G$-representation $\Gamma(\reg_{Z_\tau\cap Z_{\tilde \zeta \tau}})$ into irreducibles, the characters $\chi(N^+)$ and $\chi(N^-)$ do not occur.
\end{enumerate}
\end{lemma}

\begin{proof}
A basis of $\Gamma(\reg_{Z_\tau\cap Z_\gamma})\cong \IK[x,y]/(N^-, xy)$ is given by
$1, x,y,x^2,y^2, \dots, x^{n-1}, y^{n-1}, N^+$ which gives the first isomorphism.

A basis of $\Gamma(\reg_{Z_{\tzeta\tau}\cap Z_\gamma})\cong \IK[x,y]/(N^+, xy)$ is given by
$1, x,y,x^2,y^2, \dots, x^{n-1}, y^{n-1}, N^-$ which gives the second isomorphism.

We have $I(Z_{\tilde \zeta \tau}\cap Z_\gamma)=(N^+,N^-)=(x^n, y^n)$. Hence, a basis of $\Gamma(\reg_{Z_{\tilde \zeta \tau}\cap Z_\gamma})$ is given by
$\{x^ay^b\mid a,b<n\}$. The action of $\tzeta$ on the semi-invariant members $x^ay^a$ of the base is trivial. Hence, $\chi(x^ay^a)$ is not isomorphic to $\chi(N^\pm)$ on which $\tzeta$ acts by $-1$. The subspaces $\rho(x^ay^b)=\langle x^ay^b, x^by^a\rangle$ for $a,b<n$ and $a\neq b$ are all 2-dimensional irreducible subspaces; see \autoref{lem:Mackey}, \autoref{lem:conjcondition}, and the discussion at the beginning of
\autoref{subsect:irreven}. Hence, they are not isomorphic to $\chi(N^\pm)$ either.
\end{proof}

\begin{prop}\label{prop:sodAB}
Denoting $\cB_\tau:= \Phi_\tau\bigl(\D(V^{(\tau)})\bigr)$ and $\cB_{\tilde\zeta \tau}:= \Phi_{\tilde\zeta \tau}\bigl(\D(V^{(\tzeta\tau)})\bigr)$, we have a semi-orthogonal decomposition
\begin{equation}\label{eq:sodAB}
\Big\langle \cB_\tau(N^+), \cB_{\tilde\zeta \tau}(N^-), \cA_\gamma(Axy), \cA_\gamma(A(xy)^2)\dots, \cA_\gamma (A(xy)^{d-1}), \pi^*(\D(V/G))(A) \Big\rangle
\end{equation}
of some admissible subcategory of $\D_G(V)$.
\end{prop}

\begin{proof}
The right part of \eqref{eq:sodAB} is just the image of the semi-orthogonal sequence \eqref{eq:Asod} under the autoequivalence $\_\otimes \chi(A)$.

That $\pi^*(\D(V/G))(A)$ is left-orthogonal to the two $\cB$ pieces can be checked directly on generators, or one refers to \autoref{cor:reflectionsod}. Indeed, by \autoref{rem:muexists} and \eqref{eq:IZtau}, we can apply \autoref{cor:reflectionsod} with $\mu=\chi(N^-)$ for $g=\tau$ and with $\mu=\chi(N^+)$ for $g=\tzeta\tau$. This gives semi-orthogonal decompositions
$\langle \cB_\tau(N^-), \pi^*(\D(V/G))\rangle$ and $\langle \cB_{\tzeta\tau}(N^+), \pi^*(\D(V/G))\rangle$. Applying $\_\otimes \chi(A)$ to these two decompositions and noting that $\chi(N^\pm)\otimes \chi(A)=\chi(N^{\mp})$ gives the desired semi-orthogonality between the leftmost and rightmost parts in \eqref{eq:sodAB}.

We check the remaining semi-orthogonality conditions on generators.
Since everything is affine, generators of $\cB_\tau(N^+)$, $\cB_{\tilde\zeta \tau}(N^-)$, and $\cA_\gamma(A(xy)^c)$ are
$\Phi_{\tau}(\reg_{V^{(\tau)}})(N^+)\cong \reg_{Z_\tau}(N^+)$, $\Phi_{\tzeta\tau}(\reg_{V^{(\tzeta\tau)}})(N^-)\cong \reg_{Z_{\tzeta\tau}}(N^-)$, and $\Phi_{\gamma}(\reg_{V^{(\gamma_y)}})(A(xy)^c)\cong \reg_{Z_\gamma}(A(xy)^c)$,
respectively.

Since $\chi(N^+)^\vee\otimes \chi(N^-)\cong \chi(N^+)\otimes \chi(N^-)\cong \chi(A)$, we have \[\Ext^*_G\bigl(\reg_{\tzeta\tau}(N^+),\reg_{Z_\tau}(N^-)\bigr)\cong \Ext^*_G(\reg_{\tzeta\tau},\reg_{Z_\tau}(A))\,.\] To compute this, we apply $\Hom(\_,\reg_{Z_\tau}(A))$ to the equivariant Koszul resolution
\begin{equation} \label{eq:zetatauKoszul}
0\to \reg_V(N^+)\xrightarrow{N^+} \reg_V\to \reg_{Z_{\tilde\zeta\tau}}\to 0
\end{equation}
which gives the two term complex
\[
 \reg_{Z_\tau}(A)\xrightarrow{N^+} \reg_{Z_\tau}(N^-)
\]
whose differential is injective. Hence, by \eqref{eq:Extres}, $\Ext^i_G(\reg_{\tzeta\tau},\reg_{Z_\tau})\cong 0$ for $i=0,2$ and
\[
 \Ext^1_G(\reg_{\tzeta\tau},\reg_{Z_\tau})\cong \Gamma\bigl(\reg_{Z_{\tzeta \tau}\cap Z_\tau}(N^-)\bigr)^G
\]
which by \autoref{lem:caprep}(ii) also vanishes.
So we established $\Hom_G\bigl(\cB_{\tzeta\tau}(N^-), \cB_\tau(N^+)\bigr)=0$.

We next prove that $\Ext^*_G\bigl(\reg_{Z_\gamma}(A(xy)^c), \reg_{Z_\tau}(N^+)\bigr)\cong \Ext^*_G\bigl(\reg_{Z_\gamma}, \reg_{Z_\tau}(N^-(xy)^{d-c})\bigr)\cong 0$ for all $c$. Setting $k:=d-c$ to ease the notation, and applying $\Hom\bigl(\_, \reg_{Z_\tau}(N^-(xy)^{k})\bigr)$ to \eqref{eq:gammaKoszul}
gives
\[
 \reg_{Z_\tau}(N^-(xy)^{k})\xrightarrow{xy} \reg_{Z_\tau}(N^-(xy)^{k-1})\,.
\]
Hence,
$\Ext^i_G\bigl(\reg_{Z_\gamma}, \reg_{Z_\tau}(N^-(xy)^{k})\bigr)\cong 0$ for $i=0,2$ and
\[
\Ext^1_G\bigl(\reg_{Z_\gamma}, \reg_{Z_\tau}(N^-(xy)^{k})\bigr)\cong \Gamma\bigl(\reg_{Z_\tau\cap Z_\gamma}(N^-(xy)^{k-1})\bigr)^G\,.
\]
By \autoref{lem:caprep}(i) together with \eqref{eq:chiNrho} and $\chi(N^+)\otimes \chi(N^-)\cong \chi(A)$,
\[
\Gamma\bigr(\reg_{Z_\tau\cap Z_\gamma}(N^-(xy)^{k-1})\bigr)\cong \chi(N^-(xy)^{k-1}) \oplus \bigl(\bigoplus_{i=1}^{n-1} \rho(x^{i+n+k-1}y^{k-1})\bigr) \oplus \chi(A(xy)^{k-1})
\]
which has no $G$-invariants.
The proof that $\Ext^*_G\bigl(\reg_{Z_\gamma}(A(xy)^c), \reg_{Z_{\tzeta\tau}}(N^-)\bigr)=0$ is completely analogous using second the isomorphism of
\autoref{lem:caprep}(i) instead of the first one.
\end{proof}

\begin{remark}
Using the equivariant Koszul resolution
\begin{equation}
\label{eq:tauKoszul} 0\to \reg_V(N^-)\xrightarrow{N^-} \reg_V\to \reg_{Z_\tau}\to 0\,,
\end{equation}
one can check that $\cB_\tau$ and $\cB_{\tilde\zeta \tau}$ are both-sided orthogonal. One can go further and check that
both of $\cB_\tau$ and $\cB_{\tilde\zeta \tau}$ are both-sided orthogonal to every $\cA_\gamma((xy)^{c})$. Hence, there are several ways in which one can re-order the pieces of \eqref{eq:sodAB} such that it remains a semi-orthogonal decomposition.
\end{remark}

We denote the right-orthogonal complement of \eqref{eq:sodAB} in $\D_G(V)$ by
\[
\cC=
\Big\langle \cB_\tau(N^+), \cB_{\tilde\zeta \tau}(N^-), \cA_\gamma(Axy), \cA_\gamma(A(xy)^2)\dots, \cA_\gamma (A(xy)^{d-1}), \pi^*(\D(V/G))(A) \Big\rangle^\perp \,.
\]
We now describe a spanning class of $\cC$, which is then used in the next \autoref{subsect:exceptionaleven} to construct a full exceptional sequence of $\cC$. Merging this exceptional sequence with \eqref{eq:sodAB} will give a semi-orthogonal decomposition of the whole $\D_G(V)$ as predicted in \autoref{conj:PvdB}.

\begin{prop}\label{prop:Bspanningeven}
A spanning class of $\cC$ is given by $\cS=\{\reg_0\otimes W\mid W\in K\}$ where
\[
 K:=\irr(G)\setminus \bigl(\{\chi(A(xy)^c)\mid 0\le c<d\}\cup\{\chi(N^+), \chi(N^-) \}\bigr)\,.
\]
\end{prop}

The $W\in K$ correspond exactly to the nodes in \eqref{eq:evenMcKay} that are not in the color \textcolor{violet}{violet}.

\begin{proof}
 First of all, we have to check that $\cS$ is really contained in $\cC$, i.e.\ is right-orthogonal to all the pieces of the semi-orthogonal decomposition \eqref{eq:sodAB}.
By \autoref{lem:complementA}, $\cS$ is right-orthogonal to all the $\cA_\gamma(A(xy)^c)$ parts, and to $\pi^*(\D(V/G))(A)$.

Now, it is sufficient to check right-orthogonality to the generators $\reg_{Z_\tau}(N^+)$ and $\reg_{Z_{\tzeta\tau}}(N^-)$ of $\cB_\tau(N^+)$ and $\cB_{\tilde\zeta \tau}(N^-)$. Let $W\in \irr(G)$.
Using the Koszul resolution \eqref{eq:tauKoszul}, we get
\[
\Ext^*_G\bigl(\reg_{Z_\tau}(N^+), \reg_0\otimes W\bigr)\cong \Hom_G\bigl(\chi(N^+), W\bigr)[0]\oplus \Hom_G\bigl(\chi(N^+N^-), W\bigr)[-1]\,,
\]
which vanishes as long as $W$ is not isomorphic to  $\chi(N^+)$ or $\chi(N^+N^-)\cong \chi(A)$ which are both not contained in $K$.

Similarly, \eqref{eq:zetatauKoszul} shows that $\Ext^*_G\bigl(\reg_{Z_{\tzeta\tau}}(N^-), \reg_0\otimes W\bigr)$ is non-vanishing only for
$W$ isomorphic to  $\chi(N^-)$ or $\chi(A)$ which, again, are both not contained in $K$.

Now that we proved that $\cS\subset \cC$, we will proceed to prove that $\cS$ is a spanning class of $\cC$. A spanning class of the whole $\D_G(V)$ is given by
\[
 \Omega=\bigl\{\Ind_{G_s}^G(\reg_v\otimes U)\mid v\in V\,,\, U\in \irr(G_v)  \bigr\}\,;
\]
see \cite[Prop.\ 2.1]{BPol}.
Hence, it is sufficient to prove that $\Omega$ is contained in the triangulated subcategory $\cT\subset \D_G(V)$ generated by \begin{equation}\label{eq:gensys}\cS, \cB_\tau(N^+), \cB_{\tilde\zeta \tau}(N^-), \cA_\gamma(Axy), \cA_\gamma(A(xy)^2)\dots, \cA_\gamma (A(xy)^{d-1}), \pi^*(\D(V/G))(A) \,.\end{equation}
Let $v\in V$ be a point with trivial stabiliser $G_v=1$. Then, by \eqref{eq:Indinva},
\[
\Ind_1^G\reg_v\cong (\Ind_1^G\reg_v)(A)\cong \reg_{G\cdot v}(A)\cong (\pi^*\reg_{\pi(v)})(A)\in \pi^*(\D(V/G))(A)\,.
\]
To find the members of $\Gamma$ where $v$ has a non-trivial stabiliser, we use the following general fact several times: If a $G$-equivariant sheaf lies in the triangulated category $\cT$ and has a filtration by equivariant subsheaves such that we know for all but one graded piece that they lie in $\cT$, then \emph{a posteriori} also the one remaining piece lies in $\cT$.

We start with $v\in V^{\gamma_y}\setminus\{0\}$, hence $G_v=\langle \gamma_y\rangle$. The irreducible representations of $G_v$ are
\[
\chi_\gamma((xy)^c):=\Res_G^{\langle \gamma_y\rangle }\chi((xy)^c)\cong \Res_G^{\langle \gamma_y\rangle }\chi(A(xy)^c) \quad,\quad c=0,1,\dots, d-1.
\]
We have $\Ind_{\langle \gamma_y\rangle}^G\bigl(\reg_v\otimes \chi_\gamma((xy)^c)\bigr)\cong \reg_{G\cdot v}(A(xy)^c)\in \cA_\gamma(A(xy)^c)$. This means that we found all the
$\Ind_{\langle \gamma\rangle}^G\bigl(\reg_v\otimes \chi_\gamma((xy)^c)\bigr)$ inside $\cT$ except for $c=0$. However, by \autoref{lem:pifibres},
\[\reg_{\pi^{-1}(\pi(v))}\cong \pi^*\reg_{\pi(v)}\in \pi^*(\D(V/G))\] has a filtration whose graded pieces are \emph{all} the $\Ind_{\langle \gamma_y\rangle}^G\bigl(\reg_v\otimes \chi_\gamma((xy)^c)\bigr)$.
Since $\gamma_y$ acts trivially on $\chi(A)$, the sheaf $\pi^*\reg_{\pi(x)}(A)\in \pi^*(\D(V/G))(A)$ has a filtration with the same pieces.
Hence, we also have $\Ind_{\langle \gamma_y\rangle}^G\reg_v$
in the triangulated subcategory $\cT$ generated by \eqref{eq:gensys}.

We now consider $v\in V^\tau\setminus\{0\}$, hence $G_v=\langle \tau\rangle$. We denote by $\mathbf 1_\tau$ and $\alt_\tau$ the trivial and the non-trivial $\langle \tau\rangle$-character, respectively.
Note that $\Res_G^{\langle\tau\rangle}\chi(N^+)\cong \mathbf 1_\tau$. Hence, by \eqref{eq:Indinva},
\[\Ind_{\langle \tau\rangle}^G(\reg_v)\cong \reg_{G\cdot v}\cong \reg_{G\cdot v}(N^+)\cong \Phi_\tau(\reg_{\pi_\tau(v)})(N^+)\in \cB_\tau(N^+)\,.\]
Furthermore, by \autoref{lem:pifibres}, $\pi^*\reg_{\pi(v)}(A)\in \pi^*(\D(V/G))(A)$ has a filtration with graded pieces $\Ind_{\langle \tau\rangle}^G(\reg_v)$ and $\Ind_{\langle \tau\rangle}^G(\reg_v\otimes \alt_\tau)$. Hence, also the latter is contained in $\cT$.

One argues analogously that, for $v\in V^{\tzeta\tau}\setminus \{0\}$, hence $G_v=\langle \tzeta \tau\rangle$, the sheaves $\Ind_{\langle \tau\rangle}^G(\reg_v)$ and $\Ind_{\langle \tzeta\tau\rangle}^G(\reg_v\otimes \alt_{\tzeta\tau})$ are contained in $\cT$.

So now, we are at the point that the only members of the spanning class $\Omega$ that we still have to find inside $\cT$, the triangulated category generated by \eqref{eq:gensys}, are
\[
 \reg_0\otimes W\quad, \quad W\in \{\chi(A(xy)^c)\mid 0\le c<d\}\cup\{\chi(N^+), \chi(N^-) \}=\irr(G)\setminus K\,.
\]
By \autoref{lem:filteredspan}(i) below, together with the fact that $\chi(N^+)\otimes \chi(N^+)=\chi(1)$, the sheaf $\Phi_\tau(\reg_{\pi_\tau(0)})(N^+)\in \cB_\tau(N^+)$ has a filtration such that all graded pieces are in $\cS$, except for one, namely $\reg_0(N^+)$. Hence, also $\reg_0(N^+)\in \cT$.
Similarly, $\reg_0(N^-)\in \cT$ by \autoref{lem:filteredspan}(ii).

Hence, we now know for all $\reg_0\otimes W$ that they are contained in $\cT$ except for $\reg_0(A(xy)^c)$ with $c=0,\dots,d-1$. By \autoref{lem:Phigammaimage}, we
find that also the $\reg_0(A(xy)^c)$ are in $\cT$, with the sole possible exception of $\reg_0(A)$, which is the case $c=0$.

Finally, we invoke \autoref{lem:pifibres} again to see that $\pi^*\reg_{\pi(0)}(A)\in \pi^*(\D(V/G))(A)$ has a filtration whose graded pieces are all $\reg_0\otimes W$ for arbitrary $W\in \irr(G)$. So we found also the last missing piece of $\Omega$ in $\cT$.
\end{proof}

 \begin{lemma}\label{lem:filteredspan}
\begin{enumerate}
\item The object $\Phi_\tau(\reg_{\pi_\tau(0)})$ is a sheaf with a filtration whose graded pieces are $\reg_0((xy)^c)$, $\reg_0(N^+(xy)^c)$ for all $c=0,\dots, d-1$ and $\reg_0\otimes \rho$ for all $\rho\in \irr(G)$ with $\dim \rho=2$.
\item The object $\Phi_{\tzeta\tau}(\reg_{\pi_{\tzeta\tau}(0)})$ is a sheaf with a filtration whose graded pieces are $\reg_0((xy)^c)$, $\reg_0(N^-(xy)^c)$ for all $c=0,\dots, d-1$ and $\reg_0\otimes \rho$ for all $\rho\in \irr(G)$ with $\dim \rho=2$.
\end{enumerate}
\end{lemma}

\begin{proof}
In analogy to the proof of \autoref{lem:Phigammaimage}, we get that
$\Phi_\tau(\reg_{\pi_\tau(0)})\cong \reg_{\pi_\tau^{-1}(\pi_\tau(0))}$ has a filtration whose graded pieces are the $\reg_0\otimes W$ for those $W\in \irr(G)$ which occur in the decomposition of $\Ind_{\langle \tau\rangle}^G\mathbf 1_\tau$ into irreducibles. By adjunction, these are exactly those $W$ with $W^\tau\neq 0$. This condition holds for all $W\in \irr(G)$ except for the characters of the form $\chi(A(xy)^c)$ and $\chi(N^-(xy)^c)$.

The proof of part (ii) is the same using that the only $W\in \irr(G)$ with vanishing $\tzeta\tau$-invariants are the characters of the form $\chi(A(xy)^c)$ and $\chi(N^+(xy)^c)$.
\end{proof}

\subsection{The Exceptional Collection - even case}\label{subsect:exceptionaleven}

We now aim to construct an exceptional collection of $\cC$, the right orthogonal complement of \eqref{eq:sodAB}, which will extend \eqref{eq:sodAB} to a full semi-orthogonal decomposition of $\D_G(V)$. We already described a spanning class $\cS$ of $\cC$ in \autoref{prop:Bspanningeven} consisting of certain twisted skyscraper sheaves. More precisely, the $\reg_0\otimes W$ correspond to the nodes of \eqref{eq:evenMcKay}, which are not in the color \textcolor{violet}{violet}.

In fact, all members of $\cS$ are exceptional objects; see \autoref{cor:skyexc}. Unfortunately, it is impossible to organise all of them into an exceptional sequence. The reason is that the interior columns of \eqref{eq:evenMcKay} correspond to cycles of objects connected by morphisms of cohomological degree 2:
\[
\begin{tikzcd}
 \reg_0\otimes \rho(x^a)  \arrow[r, dotted]      & \reg_0\otimes \rho(x^a(xy)) \arrow[r, dotted]     & \reg_0\otimes \rho(x^a(xy)^2) \arrow[r, dotted]     & \dots \arrow[r, dotted]    & \reg_0\otimes \rho(x^a(xy)^{d-1}) \arrow[llll, bend left = 10, dotted]
\end{tikzcd}
 \]
However, we can remove the whole bottom row of \eqref{eq:evenMcKay}, corresponding to the objects of the form $\reg_0$ and $\reg_0\otimes \rho(x^a)$ to break these cycles. The remaining exceptional objects in $\cS$ correspond to the nodes of \eqref{eq:evenMcKay} in the usual text color black. They can be arranged into a exceptional sequence, by subsequently following the diagonal paths of solid arrows. Concretely, this means that the following is an exceptional sequence:
\begin{equation}\label{eq:skyexc}
\begin{split}
& \reg_0(xy), \reg_0\otimes \rho(x^2y), \reg_0(x^2y^2), \reg_0\otimes \rho(x^3y), \reg_0\otimes\rho(x^3y^2), \reg_0(x^3y^3),\dots,\\ & \reg_0\otimes \rho(x^{d-1}y), \dots, \reg_0\otimes \rho(x^{d-1}y^{d-1}),\rho(x^dy),\dots, \rho(x^dy^{d-1}), \dots, \rho(x^{n}y),\dots, \rho(x^{n}y^{d-1}), \\&\reg_0(N^+xy),\reg_0(N^-xy),\reg_0\otimes \rho(x^{n+1}y^2),\dots, \reg_0\otimes \rho(x^{n+1}y^{d-1}), \\&\reg_0(N^+(xy)^2), \reg_0(N^-(xy)^2), \reg_0\otimes \rho(x^{n+2}y^3),\dots, \reg_0\otimes \rho(x^{n+2}y^{d-1}), \dots,\\& \reg_0(N^+(xy)^{d-2}), \reg_0(N^-(xy)^{d-2}), \reg_0\otimes \rho(x^{n+d-2}y^{d-1}), \reg_0(N^+(xy)^{d-2}), \reg_0(N^-(xy)^{d-2})\,.
\end{split}
\end{equation}
This is not a full exceptional sequence of $\cC$ as the members of $\cS$ corresponding to the bottom row of \eqref{eq:evenMcKay} are missing. To cover also the missing pieces, we make the following definition. The definition and the ensuing lemmas are formulated for $e$ of arbitrary parity, not only for even $e$. The reason is that we will need them again later when we treat the case that $e$ is odd.

\begin{definition}\label{def:F}
For $e$ even, let $a\in \{0,\dots, n-1\}$. For $e$ odd, let $a\in \{0,\dots, n^+-1\}$. We consider the quotient of $G$-invariant ideals
\[
  F(x^a)=\frac{(x^a, y^a)}{(x^{a+1}, y^{a+1})}
  \]
of $\IK[x,y]$. Since it is a $G$-equivariant $\IK[x,y]$-module, we can consider $F(x^a)$ as an object in $\Coh_G(V)\subset \D_G(V)$.
Note that $F(x^0)=F(1)=\reg_0$.
\end{definition}

\begin{lemma}\label{lem:Fpieces}
The $G$-equivariant sheaf $F(x^a)$ has a filtration with graded pieces
\[
 \reg_0\otimes \rho(x^a),\reg_0\otimes \rho(x^ay),\reg_0\otimes \rho(x^ay^2)\dots,\reg_0\otimes \rho(x^{a}y^{a-1}), \reg_0((xy)^a)\,.
\]
Writing $a=\lambda d+k$, the pieces can be isomorphically rewritten as
\begin{align*}
\reg_0\otimes \rho(x^a), \reg_0\otimes\rho(x^ay),\dots, \reg_0\otimes\rho(x^ay^{d-1}),\\ \reg_0\otimes\rho(x^{a-d}), \reg_0\otimes\rho(x^{a-d}y), \dots, \reg_0\otimes\rho(x^{a-d}y^{d-1}),\\ \reg_0\otimes\rho(x^{a-2d}),\dots, \reg_0\otimes\rho(x^{k}), \rho(x^ky),\dots, \reg_0((xy)^k)\,.
\end{align*}
In particular, $F(x^a)\in \cC$.
\end{lemma}

\begin{proof}
We have $(x^{a+1}, y^{a+1})\supset (x,y)^{2a+1}$, hence $\supp F(x^a)=\{0\}$.
A $\IK$-basis of $F(x^a)$ is given by
$x^a, y^a, x^ay, xy^a, x^ay^2, x^2y^a, \dots, x^ay^{a-1}, x^{a-1}y^a, x^ay^a$.
Hence,
\[
\Gamma\bigl(F(x^a))\bigr)\cong \left(\bigoplus_{b=0}^{a-1}\rho(x^ay^b) \right)\oplus \chi((xy)^a)\,.
\]
Applying \autoref{lem:equifilt} gives a filtration of $F(x^a)$ with the desired graded pieces.

The rewriting of the isomorphism classes just uses that $\rho(x^\alpha y^\beta)\cong \rho(x^{\alpha+d}y^{\beta+d})$ because $(xy)^d$ is $G$-invariant.

The filtered sheaf $F(x^a)$ lies in $\cC$ because all its graded pieces are contained in $\cS\subset \cC$.
\end{proof}

\begin{remark}
We collect the graded pieces of $F(x^a)$ by starting in \eqref{eq:evenMcKay} (or, if $e$ is odd, in \eqref{eq:McKayodd}) at $x^{a}$ and following the solid arrows in the upper left direction. If an arrow leaves at the top of the diagram, it re-enters on the bottom. In other words, the graded pieces correspond to the monomials lying on the diagonal starting at $x^a$, which makes jumps from the top to bottom if $a\ge d$.

Since the $\Ext^1_G$-spaces between the graded pieces are one-dimensional, $F(x^a)$ could also be defined as the unique indecomposable filtered object with these pieces.
\end{remark}

\begin{lemma}\label{lem:FExt} Let $W\in \irr(G)$.
\begin{enumerate}
\item We have
\[
\Ext^*_G\bigl(F(x^0), \reg_0\otimes W \bigr)\cong\begin{cases}
\IK[0]\quad&\text{for $W\cong \chi(1)$,}\\
\IK[-1]\quad&\text{for $W\cong \rho(x)$,}\\
\IK[-2]\quad&\text{for $W\cong \chi(A(xy))$.}\\
0\quad&\text{else.}\\
                                           \end{cases}
 \]
 \item Let $a\in \{1,\dots, n-2\}$ for $e$ even, or $a\in \{1,\dots, n^+-1\}$ for $e$ odd. Then
\[
\Ext^*_G\bigl(F(x^a), \reg_0\otimes W \bigr)\cong\begin{cases}
\IK[0]\quad&\text{for $W\cong \rho(x^a)$,}\\
\IK[-1]\quad&\text{for $W\cong \rho(x^{a+1})$,}\\
\IK[-2]\quad&\text{for $W\cong \chi(A(xy)^{a+1})$,}\\
0\quad&\text{else.}\\
                                           \end{cases}
\]
Note that, for $e$ odd and $a=n^+-1$, we have $\rho(x^{a+1})=\rho(x^{n^+})\cong \rho(x^{n^-})$ by \autoref{lem:camparerhos} with $\lambda=\frac{e+1}2$.
\item We now consider $a=n-1$ and $e$ even. Then
\[
\Ext^*_G\bigl(F(x^{n-1}), \reg_0\otimes W\bigr)\cong\begin{cases}
\IK[0]\quad&\text{for $W\cong \rho(x^{n-1})$,}\\
\IK[-1]\quad&\text{for $W\cong \chi(N^+)$ and $W\cong \chi(N^-)$,}\\
\IK[-2]\quad&\text{for $W\cong \chi(A(xy)^n)\cong \chi(A)$,}\\
0\quad&\text{else.}\\
                                           \end{cases}
\]
\end{enumerate}
\end{lemma}

\begin{proof}
As $F(x^0)\cong \reg_0$, part (i) is just a special case of \autoref{prop:skyExt} (or an immediate consequence of \autoref{lem:skyExt}).

In order to prove part (ii), for two $G$-equivariant sheaves $F,F'$ on $V$ with zero-dimensional support, we denote their Euler bicharacteristic by
\[
 [F,F']:=\dim\Hom_G(F,F')- \dim\Ext^1_G(F,F')+ \dim\Ext^2_G(F,F')\,.
\]
Looking at the graded pieces of $F(x^a)$ described in \autoref{lem:Fpieces} and counting the arrows in \eqref{eq:evenMcKay}  (or, if $e$ is odd, in \eqref{eq:McKayodd}) gives for $W\in \irr(G)$ the formula
\begin{equation}\label{eq:FEuler}
 \bigl[F(x^a), \reg_0\otimes W \bigr]=\begin{cases}
1\quad&\text{for $W\cong \rho(x^a)$,}\\
-1\quad&\text{for $W\cong \rho(x^{a+1})$,}\\
1\quad&\text{for $W\cong \chi(A(xy)^{a+1})$,}\\
0\quad&\text{else.}
\end{cases}
\end{equation}
Indeed, $\reg_0\otimes \rho(x^a)$ is among the graded pieces of $F(x^a)$ and it has the identity morphism to itself, but there are no arrows from any of the other graded pieces to $\reg_0\otimes \rho(x^a)$. This gives
\begin{align*}
\bigl[F(x^a), \reg_0\otimes \rho(x^a)]&=\sum_{b=0}^{a-1} \bigl[\reg_0\otimes \rho(x^ay^b), \reg_0\otimes \rho(x^a)\bigr] + \bigl[\reg_0((xy)^a), \reg_0\otimes \rho(x^a)\bigr]\\
&=1+0+\dots+0=1\,.
\end{align*}
Similarly, there is one solid arrow $x^a\to x^{a+1}$ in \eqref{eq:evenMcKay} and \eqref{eq:McKayodd}, standing for a morphism $\reg_0\otimes \rho(x^a)\to \reg_0\otimes \rho(x^{a+1})[1]$ of cohomological degree $1$, but no arrows from any of the other graded pieces to $\rho(x^{a+1})$. This gives $\bigl[F(x^a), \reg_0\otimes \rho(x^{a+1}) \bigr]=-1$.

There is one dotted arrow $(xy)^{a}{\scriptscriptstyle \cdots>} A(xy)^{a+1}$, standing for a morphism of cohomological degree 2. Note that this arrow is visible only in \eqref{eq:evenMcKay}. In \eqref{eq:McKayodd}, we had to leave out the leftmost column with the $A(xy)^c$ in order to fit the diagram on one page. But we
still have the non-vanishing $\Ext^2_G\bigl(\reg_0((xy)^a), \reg_0(A(xy)^{a+1}\bigr)\cong \IK$ independent of the parity of $e$; see \autoref{prop:skyExt}. There are no arrows from other graded pieces of $F(x^a)$ to $\reg_0(A(xy)^{a+1})$. Hence, $\bigl[F(x^a), \reg_0(A(xy)^{a+1}) \bigr]=-1$

For the vanishing of all the other $\bigl[F(x^a), \reg_0\otimes W \bigr]$, one has to go through three cases. The first is that $W\cong \rho(x^ay^b)$ with $0<b< a$ or $W\cong \chi((xy)^a)$ so that $\reg_0\otimes W$ is among the graded pieces of $F(x^a)$. Then there is the identity on $\reg_0\otimes W$, a solid arrow $x^ay^{b-1}\to x^ay^b$ (in the case $W\cong \chi((xy)^a)$, we set $b=a$), and no other arrows in \eqref{eq:evenMcKay} or \eqref{eq:McKayodd} from graded pieces of $F(x^a)$ to $x^ay^b$. Hence,
$\bigl[F(x^a), \reg_0\otimes W\bigr]=1-1=0$.

Similarly, let $W\cong \rho(x^{a+1}y^b)$ with $0< b< a+1$, which corresponds to a node in the diagonal of \eqref{eq:evenMcKay} or \eqref{eq:McKayodd} directly right of the diagonal representing the graded pieces of $F(x^a)$. Then there is a solid arrow $x^ay^b\to x^{a+1}y^b$, a dotted arrow $x^ay^{b-1}{\scriptscriptstyle \cdots>} x^{a+1}y^b$, and no other arrows from graded pieces of $F(x^a)$ to $x^{a+1}y^b$. Hence, $\bigl[F(x^a), \reg_0\otimes \rho(x^ay^b)\bigr]=-1+1=0$.

In all other cases, there are no arrows from graded pieces of $F(x^a)$ to $\reg_0\otimes W$ at all, so \eqref{eq:FEuler} is established.

The $\IK[x,y]$-module $F(x^a)$ is generated by $x^a$ and $y^a$. So a homomorphism from $F(x^a)$ is determined by the image of these two elements whose $\IK$-linear span is $\rho(x^a)$. Hence,
\begin{equation}\label{eq:FHoms}
\Hom_G\bigl(F(x^a), \reg_0\otimes W \bigr)\cong\begin{cases}
\IK \quad\text{for $W\cong \rho(x^a)$,}\\
0\quad\text{else.}
                                           \end{cases}
\end{equation}
The socle of $F(x^a)$ is $\Ann_{(x,y)}F(x^a)=\langle x^a y^a\rangle =\chi((xy)^a)$. Hence,
\[
\Hom_G\bigl(\reg_0\otimes W , F(x^a)\bigr)\cong\begin{cases}
\IK \quad\text{for $W\cong \chi((xy)^a)$,}\\
0\quad\text{else.}
                                           \end{cases}
\]
By Serre duality, this translates to
\begin{equation}\label{eq:FExt2}
\Ext^2_G(F(x^a), \reg_0\otimes W )\cong\begin{cases}
\IK \quad\text{for $W\cong \chi(A(xy)^{a+1})$,}\\
0\quad\text{else.}
                                           \end{cases}
\end{equation}
The three equations \eqref{eq:FEuler}, \eqref{eq:FHoms}, and \eqref{eq:FExt2} together prove part (ii) of the assertion.

The proof of part (iii) is completely analogous, with the only difference that for $e$ even and $a=n-1$, instead of a solid arrow $x^a\to x^{a+1}$, there are solid arrows $x^{n-1}\to N^+$ and $x^{n-1}\to N^-$ in \eqref{eq:evenMcKay}.
\end{proof}

\begin{prop}[$e$ even] \label{prop:Cseqeven} The following is a full exceptional sequence of $\cC$:
\begin{equation}\label{eq:Ceseseven}
\bigl(\text{exceptional sequence \eqref{eq:skyexc}} \bigr)\,,\, F(x^0)\,,\, F(x)\,,\, \dots\,,\, F(x^{n-1})\,.
\end{equation}
\end{prop}

\begin{proof}
 Let us first show that all $F(x^a)$ for $a=0, \dots,n-1$ are exceptional objects.
By \autoref{lem:Fpieces} and \autoref{lem:FExt}, the only graded piece of $F(x^a)$ such that $\Ext^*_G$ from $F(x^a)$ to it is non-vanishing is $\reg_0\otimes \rho(x^a)$. Hence,
\[
 \Ext^*_G\bigl(F(x^a), F(x^a)\bigr)\cong  \Ext^*_G\bigl(F(x^a), \reg_0\otimes \rho(x^a) \bigr)\cong \IK[0]\,.
\]

For the semi-orthogonality, again by \autoref{lem:FExt}, the only sheaves of the form $\reg_0\otimes W$ with $W\in \irr(G)$ which are not in the right-orthogonal complement of all the $F(x^a)$ are those with $W$ isomorphic to one of $\chi(1)$, $\rho(x^a)$, $\chi(A(xy)^c)$, $\chi(N^+)$, $\chi(N^{-})$. None of these $\reg_0\otimes W$ occurs in the exceptional sequence \eqref{eq:skyexc} (the corresponding nodes in \eqref{eq:evenMcKay} all have another color than black).

\autoref{lem:Fpieces} and \autoref{lem:FExt} also say that, for $n-1\ge i>j\ge 0$ and $\reg_0\otimes W$ any graded piece of $F(x^j)$, we have $\Ext^*_G\bigl(F(x^i), \reg_0\otimes W\bigr)\cong 0$. Hence
\[
\Ext^*_G\bigl(F(x^i), F(x^j)\bigr)\cong 0 \quad\text{for } n-1\ge i>j\ge 0\,.
\]
To prove fullness of the sequence \eqref{eq:Ceseseven}, we check that the spanning class $\cS$ of \autoref{prop:Bspanningeven} is contained in $\cU$, the triangulated category generated by \eqref{eq:Ceseseven}. The only members of $\cS$ not already contained in \eqref{eq:skyexc} are $\reg_0$, $\reg_0\otimes \rho(x)$, \dots, $\reg_0\otimes \rho(x^{n-1})$. We have $\reg_0\cong F(x^a)$, so this is a member of \eqref{eq:Ceseseven}. Since, for $0< a<n$, all graded pieces of $F(x^a)$ are contained in \eqref{eq:Ceseseven} except for $\reg_0\otimes \rho(x^a)$, also the latter lies in $\cU$.
\end{proof}

\begin{proof}[Proof of \autoref{thm:main} for $e$ even]
Combing \autoref{prop:sodAB} and \autoref{prop:Cseqeven}, we get the semi-orthogonal decomposition
\[
\D_G(V)=
\Big\langle \text{(exc.\ seq.\ \eqref{eq:Ceseseven})},\cB_\tau(N^+), \cB_{\tilde\zeta \tau}(N^-), \cA_\gamma(Axy),\dots, \cA_\gamma (A(xy)^{d-1}), \pi^*(\D(V/G))(A) \Big\rangle\,.
\]
Comparing this with \autoref{lem:conjeven} confirms \autoref{conj:PvdB} for $G(m,e,2)$ with $e$ even. Indeed, the $\cB$ and $\cA$ pieces are in bijection with the $d+1$ conjugacy classes of reflections, and each piece is equivalent to the derived category of the quotient of the fixed point locus of the corresponding reflection. Furthermore, the members of \eqref{eq:Ceseseven} are in bijection to $K$, which is $\irr(G)$ with $d+2$ members removed; see \autoref{prop:Bspanningeven}. Hence, length of \eqref{eq:Ceseseven} is equal to the number of conjugacy classes of elements of $G$ whose fixed point locus is $\{0\}$. To conclude, note that the category spanned by an exceptional object is equivalent to the derived category of a point.
\end{proof}

\subsection{The Component Associated to $\tau$ - Odd Case}\label{subsect:tauodd}

For this and the next subsection, we assume that $e$ is odd.
In contrast to the even case, for $e$ odd, the reflections $\tzeta^i\tau$ form just one, not two, conjugacy classes. We consider $\tau$ as a representative of this class.

As in the even case, we can define $Z_\tau:=G\cdot V^\tau\subset V$ equipped with the reduced subscheme structure, and write the embedding as $\iota_\tau\colon Z_\tau\hookrightarrow V$. However, the analogue of \autoref{prop:tauquot} fails, which means that \autoref{ass:g} is not satisfied. To see this, recall that $\reg(V^\tau)^G\cong \IK[t^d]$ where $t=[x]=[y]$; see \eqref{eq:tauoddinva}.
Hence, the restriction map
\begin{equation}\label{eq:nudescription}
 \reg(Z_\tau)^G\cong \left( \frac{\IK[x,y]}{x^m-y^m} \right)^G\to \reg(V^{\tau})^{\CC(\tau)}\cong \IK[t^{d}]
\end{equation}
is not surjective, as the generators $(xy)^d$ and $\frac12(x^m+y^m)$ of the ring of invariants get mapped to $t^{2d}$ and $t^m=t^{de}$, respectively. Hence, setting $u:=t^d$, we see that the morphism $\nu_\tau\colon V^{(\tau)}\to Z_\tau/G$
is not an isomorphism, but the normalisation of the cuspidal curve $Z_\tau/G\cong \Spec\IK[u^2,u^e]$.

In order to construct a well-behaved functor $\Phi_\tau\colon\D(V^{(\tau)})\to \D_G(V)$ such that the objects in its image are supported on $Z_\tau$, we consider the reduced fibre product
\[
 \hZ_\tau:=\bigl(V^{(\tau)}\times_{V/ G} Z_\tau \bigr)_{\mathsf{red}}\subset V^{(\tau)}\times Z_\tau \subset V^{(\tau)}\times V\,.
\]
The constituting morphisms for the fibre product are the quotient morphism $\pi\colon V\to V/G$ and the canonical morphism $\nu_\tau\colon V^{(\tau)}\to Z_\tau/G$.

We write $\bar 0$ for the point in $V^{(\tau)}$ corresponding to the fixed point $0\in V^{\tau}$. This corresponds to the point $u=0$ under the isomorphism $V^{(\tau)}\cong \Spec\IK[u]$.
The discussion above shows that $\nu_\tau\colon V^{(\tau)}\to Z_\tau/G$ is a bijection and an isomorphism away from $\bar 0$; compare also \cite[Prop.\ 2.2.4(i)]{PvdB}. Hence, also the projection $p_{Z}\colon\hZ_\tau\to Z_\tau$ is a bijection and an isomorphism away from the origin.
Furthermore, as $\widehat Z_\tau$ is a reduced curve, hence Cohen--Macaulay, and $V^{(\tau)}\cong \IA^1$ is regular, miracle flatness \cite[Ex.\ 18.17]{Eisenbud--commutativebook} gives that the projection
\[
\widehat\pi_\tau\colon \hZ_\tau\to V^{(\tau)}
\]
is flat. Considering $\hZ_\tau$ as a $G$-invariant closed subscheme of $V^{(\tau)}\times V$ (with $G$ acting trivially on the first factor $V^{(\tau)}$), we get the equivariant Fourier--Mukai transform
\[
 \Psi_\tau:=\FM_{\reg_{\widehat Z_\tau}}\colon \D(V^{(\tau)})\to \D_G(V)\,.
\]
For generalities on equivariant Fourier--Mukai transforms, see e.g.\ \cite[Sect.\ 2]{Plo}, \cite[Sect.\ 2.1]{Krug--NakaP}. However, we will not need any general theory, as by projection formula, we can rewrite $\Psi_\tau$ as
\[
\Psi_\tau\cong \iota_{\tau*}\circ p_{Z*}\circ  \widehat\pi_\tau^*\circ \triv\,.
\]
None of the four functors of the composition needs to be derived, as the push-forwards are along affine morphisms, and the pull-back is along a flat morphism. For every $t\in V^{(\tau)}$,
\begin{equation}\label{eq:psiskyimage}
\Psi_\tau(\reg_t)\cong \reg_{\widehat \pi_\tau^{-1}(t)}
\end{equation}
where $\widehat \pi_\tau^{-1}(t)$ denotes the scheme-theoretic preimage of $\{t\}$ under $\widehat\pi_\tau$, which we consider as a subscheme of $V$ via the identification $\{t\}\times V\cong V$.

\begin{lemma}\label{lem:xiinva}
For every $t\in V^{(\tau)}$, we have
\begin{equation}\label{eq:Gammatau}
\Gamma(\reg_{\widehat \pi_\tau^{-1}(t)})\cong \left(\bigoplus_{c=0}^{d-1}\chi((xy)^c)\right)\oplus \left(\bigoplus_{\rho\in\irr(G),\,\dim \rho=2}\rho\right)\,.
\end{equation}
\end{lemma}
In other words, every irreducible $G$-representation occurs with multiplicity 1 in $\Gamma(\reg_{\widehat \pi_\tau^{-1}(t)})$ except for the characters $\chi(A(xy)^c)$.

\begin{proof}
Flatness of $\widehat \pi_\tau$ implies that the $G$-representation $\Gamma(\reg_{\widehat \pi_\tau^{-1}(t)})$ is independent of $t$; see \autoref{lem:flatfibrerep}.
For $t\neq \bar 0$, the fibre $\widehat \pi_\tau^{-1}(t)$ is a reduced orbits $G\cdot v\subset V$ with $v\in V^{\tau}\setminus\{0\}$. For such $v$, we have $G_v=\langle \tau\rangle$, hence
\[
\Gamma(\reg_{G\cdot v})\cong \Ind_{\langle \tau\rangle}^G\mathbf 1_\tau\,.
\]
By adjunction, $W\in \irr(G)$ occurs in the decomposition of $\Ind_{\langle \tau\rangle}^G\mathbf 1_\tau$ into irreducibles with multiplicity $\dim W^\tau$, which gives the right-hand side of \eqref{eq:Gammatau}.
\end{proof}

\begin{prop}[$e$ odd]\label{prop:Psiff}
 The functor $\Psi_\tau\colon \D(V^{(\tau)})\to \D_G(V)$ has a left and a right adjoint and is fully faithful.
\end{prop}

\begin{proof}
The argument for the existence of the adjoints is the same as in the proof of \autoref{prop:reflectionff}.
For the fully faithfulness, we use the Bondal--Orlov criterion \cite[Thm.\ 1.1]{BO--ffcrit}, \cite[Thm.\ 1.2]{BPol} which allows to check fully faithfulness of Fourier--Mukai transforms on skyscraper sheaves of points. Concretely, the functor $\Psi_\tau$ is fully faithful if and only if the following three conditions hold:
\begin{enumerate}
 \item\label{BO1} $\Hom_G(\Psi_\tau(\reg_t), \Psi_\tau(\reg_t))\cong \IK$ for every $t\in V^{(\tau)}$,
\item\label{BO2} $\Ext^i_G(\Psi_\tau(\reg_t), \Psi_\tau(\reg_t))\cong 0$ for every $t\in V^{(\tau)}$ and every $i\notin\{0,1=\dim V^{(\tau)}\}$,
\item\label{BO3} $\Ext^*_G(\Psi_\tau(\reg_s), \Psi_\tau(\reg_t))\cong 0$ for every $s,t\in V^{(\tau)}$ with $s\neq t$.
 \end{enumerate}
Recall that $\Psi_\tau(\reg_t)\cong \reg_{\widehat \pi_\tau^{-1}(t)}$; see \eqref{eq:psiskyimage}. We have
\[
 \Hom_G(\reg_{\widehat \pi_\tau^{-1}(t)}, \reg_{\widehat \pi_\tau^{-1}(t)})\cong \Gamma(\reg_{\widehat \pi_\tau^{-1}(t)})^G\cong \IC
\]
where the last isomorphism is due to \autoref{lem:xiinva}. This confirms Condition \ref{BO1}. We also have
\[
 \Hom_G(\reg_{\widehat \pi_\tau^{-1}(t)}, \reg_{\widehat \pi_\tau^{-1}(t)}(Axy))\cong \Gamma(\reg_{\widehat \pi_\tau^{-1}(t)}(Axy))^G\,.
\]
The invariants on the right side vanish because $\chi(Axy)^\vee\cong \chi(A(xy)^{d-1})$ does not occur in \eqref{eq:Gammatau}. By Serre duality \eqref{eq:Serre}, this implies
$\Ext^2_G(\reg_{\widehat \pi_\tau^{-1}(t)}, \reg_{\widehat \pi_\tau^{-1}(t)})\cong 0$ confirming Condition \ref{BO2}.
Condition \ref{BO3} holds because for $s,t\in V^{(\tau)}$ with $s\neq t$, the supports of $\Psi_\tau(\reg_s)$ and $\Psi_\tau(\reg_t)$ are different $G$-orbits in $V$.
 \end{proof}

  \begin{remark}
In the case that $e$ is even, we can still define $\widehat Z_\tau=\bigl(V^{(\tau)}\times_{V/ G} Z_\tau \bigr)_{\mathsf{red}}$
as the reduced fibre product and $\Psi_\tau$ as the Fourier--Mukai transform with kernel $\widehat Z_\tau$. However, this gives nothing new, because the projection $\widehat Z_\tau\to Z_\tau$ is an isomorphism for $e$ even. Hence, $\Psi_\tau \cong \Phi_\tau$ agrees with the fully faithful functor of \autoref{prop:Phitau}.
 \end{remark}

\begin{prop}[$e$ odd]\label{prop:sodABodd}
We have a semi-orthogonal decomposition
\begin{equation}\label{eq:sodABodd}
\bigl\langle \cA_\gamma(A(xy)), \cA_\gamma(A(xy)^2)\dots, \cA_\gamma (A(xy)^{d-1}), \pi^*(\D(V/G))(A), \cB_\tau(A) \bigr\rangle
\end{equation}
of some admissible subcategory of $\D_G(V)$ where $\cB_\tau:= \Psi_\tau\bigl(\D(V^{(\tau)})\bigr)$.
\end{prop}

\begin{proof}
The left part of \eqref{eq:sodABodd} is just $\_\otimes \chi(A)$ applied to the semi-orthogonal decomposition \eqref{eq:Asod}. By
\autoref{prop:Psiff}, all that is left to show is that $\cB_\tau(A)$ is left-orthogonal to the other pieces. We will check that on the spanning  class
\begin{equation}\label{eq:Bspanning}
\bigl\{ \Psi_\tau(\reg_t)(A)\cong \reg_{\widehat \pi_\tau^{-1}(t)}(A) \mid t\in V^{(\tau)} \bigr\}
\end{equation}
 of $\cB_\tau(A)$, and the generators $\reg_{Z_\gamma}(A(xy)^c)$ and $\reg_V(A)$ of the other pieces.
By Serre duality,
\[
\Ext^*_G\bigl(\reg_{\widehat \pi_\tau^{-1}(t)}(A), \reg_V(A)\bigr)\cong   \Ext^*_G\bigl(\reg_{\widehat \pi_\tau^{-1}(t)}, \reg_V\bigr)\cong \Ext^*_G\bigl(\reg_V,\reg_{\widehat \pi_\tau^{-1}(t)}(Axy)\bigr)^\vee[-2]\,.
\]
We have $\Ext^*_G\bigl(\reg_V,\reg_{\widehat \pi_\tau^{-1}(t)}(Axy)\bigr)\cong \Gamma\bigl(\reg_{\widehat \pi_\tau^{-1}(t)}(Axy)\bigr)^G[0]$, which vanishes by \autoref{lem:xiinva}. Hence, we also have the desired vanishing of $\Ext^*_G\bigl(\reg_{\widehat \pi_\tau^{-1}(t)}(A), \reg_V(A)\bigr)$.

For $t\neq \bar 0$, the sheaves $\reg_{\widehat \pi_\tau^{-1}(t)}(A)$ and $\reg_{Z_\gamma}(A(xy)^c)$ have disjoint support, hence are orthogonal in $\D_G(V)$.
So now, let $t=\bar 0$. To ease the notation, we write
$\xi:=\widehat \pi_\tau^{-1}(\bar 0)\subset V$. It is a non-reduced subscheme concentrated in $0\in V$.
We need to check the vanishing of
\begin{align*}
 \Ext^*_G\bigl(\reg_\xi(A), \reg_{Z_\gamma}(A(xy)^c)\bigr)\cong \Ext^*_G\bigl(\reg_\xi, \reg_{Z_\gamma}((xy)^c)\bigr)&\cong \Ext^*_G\bigl( \reg_{Z_\gamma}((xy)^c), \reg_\xi(Axy)\bigr)^\vee[-2]\\& \cong \Ext^*_G\bigl( \reg_{Z_\gamma}, \reg_\xi(A(xy)^{d+1-c})\bigr)^\vee[-2]
\end{align*}
where the second isomorphism is again Serre duality. To ease the notation, we set $k:=d+1-c$. To compute $\Ext^*_G\bigl( \reg_{Z_\gamma}, \reg_\xi(A(xy)^{k})\bigr)$ we apply $\Hom\bigl(\_, \reg_{\xi}(A(xy)^k)\bigl)$ to the resultion \eqref{eq:tauKoszul}. This gives the two term complex
\[
\reg_\xi(A(xy)^k)\xrightarrow{xy} \reg_\xi(A(xy)^{k-1})
\]
whose cohomologies we abbreviate by $\cH^0:=\ker(xy)$ and $\cH^1:=\coker(xy)$.
Then $\Gamma(\cH^0)$ is a subrepresentation of $\Gamma(\reg_\xi)(A(xy)^k)$ and $\Gamma(\cH^1)$ is a quotient of the $G$-representation $\Gamma(\reg_\xi)(A(xy)^{k-1})$. As, by \autoref{lem:xiinva}, no character of the form $\chi(A(xy)^c)$ occurs in $\Gamma(\reg_\xi)$, we have
$\Gamma(\cH^0)^G\cong 0\cong \Gamma(\cH^1)^G$, which by \eqref{eq:Extres} gives the desired vanishing.
\end{proof}

Before we can provide a spanning class of the orthogonal complement of
\eqref{eq:sodABodd} in analogy to \autoref{prop:Bspanningeven}, we need a more concrete description of $\xi=\widehat\pi_\tau^{-1}(\bar 0)\subset V$.

\begin{lemma}\label{lem:xi}
The vanishing ideal of $\xi\in V$ is $I(\xi)=\bigl((xy)^d, x^{n^+}, y^{n^+}\bigr)$
\end{lemma}

\begin{proof}
Let $I(\hZ_{\tau})$ be the vanishing ideal of $\hZ_\tau\subset V^{(\tau)}\times V=\Spec\IK[u,x,y]$. We check that
\begin{equation}\label{eq:ideal}
I(\widehat Z_{\tau})\supset \bigl(x^m-u^e,y^m-u^e, (xy)^d-u^{2}, x^{n^+}-uy^{n^-}, y^{n^+}-ux^{n^-}   \bigr)\,.
\end{equation}
That the three relations $x^m=u^e$, $y^m=u^e$, $(xy)^d=u^2$ hold on the fibre product $V^{(\tau)}\times_{V/G}V$ follows by the description \eqref{eq:nudescription} of $\nu_\tau$. As $\widehat Z_{\tau}$ is defined as the reduced fibre product, it is sufficient to check that the other two relations $x^{n^+}=uy^{n^-}$ and $y^{n^+}-ux^{n^-}$ hold generically on $Z_\tau$, so we may assume that $u$ is invertible. We write $e=2k+1$ which gives $n^-=kd$ and $n^+=(k+1)d$, and compute
\begin{equation}\label{eq:u}
 u=\frac{u^e}{(u^2)^k}=\frac{x^m}{(xy)^{dk}}=\frac{x^m}{x^{n^-}y^{n^-}}=\frac{x^{n^+}}{y^{n^-}}\,,
\end{equation}
which gives $x^{n^+}=uy^{n^-}$. When we replace $u^e$ by $y^m$ instead of by $x^m$ in the second step of \eqref{eq:u}, we get
$y^{n^+}-ux^{n^-}$.

Actually, there should be equality in \eqref{eq:ideal}, but this is harder to check than the inclusion and we can avoid it with the following argument.
Setting $u=0$ in \eqref{eq:ideal} gives
\begin{equation}\label{eq:inclusion}
I(\xi)\supset \bigl((xy)^d, x^{n^+}, y^{n^+}\bigr)=:J
\end{equation}
A $\IK$-base of $\IK[x,y]/J$ is given by
\[
 \bigl\{x^ay^b\mid \min\{a,b\}<d\,,\, \max\{a,b\}<n^+ \bigr\}\,.
\]
Counting the members of the base gives $\dim (\IK[x,y]/J)=md$.
From the proof of \autoref{lem:xiinva}, we see that this agrees with
\[
 \dim\Gamma(\reg_\xi)=\dim(\Ind_{\langle \tau\rangle}^G\mathbf 1)=\frac{|G|}{|\langle \tau\rangle|}=\frac{2md}2=md\,.
\]
Hence, the inclusion \eqref{eq:inclusion} must be an equality.
\end{proof}

\begin{lemma}\label{lem:xiExtvanish}
 We have $\Ext^*_G(\reg_\xi(A), \reg_0\otimes W)\cong 0$ for all
 \[
W\in K:=\irr(G)\setminus \bigl(\{\chi(A(xy)^c)\mid 0\le c<d\}\cup\{\rho(x^{n^-})\}\bigr)\,.
 \]
\end{lemma}

The $W\in K$ correspond exactly to the nodes in \eqref{eq:McKayodd} after removing the \textcolor{violet}{violet $x^{n^-}$}.

\begin{proof}
We have $\chi((xy)^c)\otimes \chi(A)\cong \chi(A(xy)^c)$ and $\rho\otimes \chi(A)\cong \rho$ for $\rho\in \irr(G)$ with $\dim\rho=2$.
Hence, by \autoref{lem:xiinva} together with \autoref{lem:equifilt}, the equivariant sheaf $\reg_\xi(A)$ has a filtration such that every $\reg_0\otimes U$ with $U\in K\cup\{\rho(x^{n^-})\}$, i.e.\ every $U\in \irr(G)$ which is not a character of the form $\chi(A(xy)^c)$ occurs exactly once as a graded piece. These are exactly the irreducible representations displayed in \eqref{eq:McKayodd} (we left out the column with the $A(xy)^c$ to be able to fit \eqref{eq:McKayodd} on one page).

We compute that, for every $W\in K$, we have the vanishing of the Euler characteristic
\begin{equation}\label{eq:xiEulervanish}
[\reg_\xi, \reg_0\otimes W]=0\,.
\end{equation}
Indeed, for $c=0,\dots, d-1$, we have exactly one solid arrow into $(xy)^c$ in \eqref{eq:McKayodd}. Since we also have the identity on $\reg_0((xy)^c)$, this gives $[\reg_\xi, \reg_0((xy)^c)]=1-1=0$.

For $\rho\in \irr(G)$ of $\dim\rho=2$, we have the identity on $\reg_0\otimes \rho$, two solid arrows, and one dotted arrow into the corresponding monomial in \eqref{eq:McKayodd}. This gives $[\reg_\xi, \reg_0\otimes \rho]=1-2+1=0$, so we confirmed \eqref{eq:xiEulervanish}.

As $1$ is a generator of $\reg_\xi$ as a $\reg_V$-module, we have $\Hom_G\bigl(\reg_\xi(A), \reg_0\otimes W\bigr)\cong 0$ for $W\in \irr(G)$, except for $W\cong \chi(A)$.
Furthermore, by \autoref{lem:xi}, the socle of $\reg_\xi$ is
\[
 \Ann_{(x,y)}\reg_\xi=\langle x^{n^+-1}y^{d-1}, x^{d-1}y^{n^+-1}\rangle \cong \rho(x^{n^+-1}y^{d-1})\,.
\]
Hence, by Serre duality,
\[
 \Ext^2_G\bigl(\reg_\xi(A), \reg_0\otimes W\bigr)^\vee\cong \Hom_G\bigl(\reg_0\otimes W, \reg_\xi(xy)\bigr)\cong \Hom_G\bigl(\reg_0\otimes W((xy)^{d-1}), \reg_\xi\bigr)
\]
vanishes except for $W((xy)^{d-1})\cong \rho(x^{n^+-1}y^{d-1})$, which is the same as $W\cong \rho(x^{n^-})$ by \eqref{eq:chirho}. As $\chi(A), \rho(x^{n^-})\notin K$, we have proved
\[
 \Hom_G\bigl(\reg_\xi(A), \reg_0\otimes W\bigr)\cong 0\cong \Ext^2_G\bigl(\reg_\xi(A), \reg_0\otimes W\bigr) \quad\text{for every }W\in K\,.
\]
Together with \eqref{eq:xiEulervanish}, this gives the assertion.
\end{proof}

We denote the right-orthogonal complement of \eqref{eq:sodABodd} in $\D_G(V)$ by
\[
\cC=
\bigl\langle \cA_\gamma(A(xy)), \cA_\gamma(A(xy)^2)\dots, \cA_\gamma (A(xy)^{d-1}), \pi^*(\D(V/G))(A), \cB_\tau(A) \bigr\rangle^\perp \,.
\]

\begin{prop}[$e$ odd]\label{prop:Bspanningodd}
A spanning class of $\cC$ is given by $\cS=\{\reg_0\otimes W\mid W\in K\}$ where
\[
 K=\irr(G)\setminus \bigl(\{\chi(A(xy)^c)\mid 0\le c<d\}\cup\{\rho(x^{n^-})\}\bigr)\,.
\]
\end{prop}

\begin{proof}
 This is very similar to the prove of \autoref{prop:Bspanningeven}. We first check that $\cS\subset \cC$. That $\cS$ is right orthogonal to
$\bigl\langle \cA_\gamma(A(xy)), \cA_\gamma(A(xy)^2)\dots, \cA_\gamma (A(xy)^{d-1}), \pi^*(\D(V/G))(A) \bigr\rangle$ follows directly from \autoref{lem:complementA}.

That $\cS$ is right orthogonal to $\cB_\tau(A)$ can be checked on the spanning class \eqref{eq:Bspanning}. So we have to verify that
\[\Ext^*_G\bigl(\reg_{\widehat \pi_\tau^{-1}(t)}, \reg_0\otimes W\bigr)\cong 0\quad \text{for every $t\in V^{(\tau)}$ and every $W\in K$.}\]
For $t\neq \bar 0$, the supports of $\reg_{\widehat \pi_\tau^{-1}(t)}$ and
$\reg_0\otimes W$ are disjoint, so these two sheaves are orthogonal. For $t=\bar 0$, the desired vanishing is \autoref{lem:xiExtvanish}.

So now we established $\cS\subset \cC$. That $\cS$ is a spanning class of $\cC$ is again shown by checking that spanning class
\[
 \Omega=\bigl\{\Ind_{G_v}^G(\reg_x\otimes U)\mid v\in V\,,\, U\in \irr(G_v)  \bigr\}
\]
of $\D_G(V)$
is contained in the triangulated subcategory $\cT\subset \D_G(V)$ spanned by $\cS$ and \eqref{eq:sodABodd}.
In the two cases $G_v=1$ and $v\in Z_\gamma\setminus \{0\}$, which means $G_v=\langle \gamma_y\rangle$, it was already shown in the proof of \autoref{prop:Bspanningeven} that all the $\Ind_{G_v}^G(\reg_x\otimes U)$ for $U\in \irr(G_v)$ are contained in \eqref{eq:Asod} twisted by $\chi(A)$. This part is independent of the parity of $e$, so we have these $\Ind_{G_v}^G(\reg_v\otimes U)$ in $\cT$.

Let now $v\in V^\tau\setminus \{0\}$. Then $G_v=\langle \tau\rangle$,  which has irreducible representations $\mathbf 1_\tau$ and $\alt_\tau=\Res\chi(A)$. By \eqref{eq:Indprojformula}, we have
\[
 \Ind_{\langle \tau\rangle}^G(\reg_v\otimes \alt_\tau)\cong \reg_{G\cdot v}(A)\in \cB_\tau(A)\,.
\]
By \autoref{lem:pifibres}, the sheaf $\reg_{\pi^{-1}(\pi(v))}(A)\cong \pi^*\reg_{\pi(v)}(A)\in \pi^*(\D(V/G))(A)$ has a filtration with graded pieces $\Ind_{\langle \tau\rangle}^G(\reg_v\otimes \alt_\tau)$ and $\Ind_{\langle \tau\rangle}^G(\reg_v)$. Hence, also the latter is contained in $\cT$.

It remains to find the $\reg_0\otimes W$ with $W\in \irr(G)$ in $\cT$. For $W\in K$, they are in $\cS\subset \cT$ by definition.

To find $\reg_0\otimes \rho(x^{n^-})$ in $\cT$, let $\eta:=\pi^{-1}([0])\subset V$ be the scheme-theoretic fibre of the quotient morphism $\pi\colon V\to V/G$, which means that $I(\eta)=((xy)^d, x^m+y^m)$; see the beginning of \autoref{subsect:irrep}. There is a surjective morphism $\reg_{\eta}\to \reg_{\xi}$ whose kernel we denote by $K$. We have
\[
\Gamma(\reg_{\eta})\cong\IK\langle G\rangle\cong \left(\bigoplus_{c=0}^{d-1}\chi((xy)^c)  \right)\oplus\left(\bigoplus_{c=0}^{d-1}\chi(A(xy)^c)  \right)\oplus\left(\bigoplus_{\rho\in \irr(G)\,,\, \dim \rho=2}\rho^{\oplus 2}  \right)\,;\]
see \autoref{lem:pifibres}.
Comparing this to \autoref{lem:xiinva} gives
\[
\Gamma(K)\cong \left(\bigoplus_{c=0}^{d-1}\chi(A(xy)^c)  \right)\oplus\left(\bigoplus_{\rho\in \irr(G)\,,\, \dim \rho=2}\rho  \right)\,.
\]
Hence, by \autoref{lem:equifilt} the equivariant sheaf $K(A)=K\otimes \chi(A)$ has a filtration whose graded pieces are $\reg_0((xy)^c)$ for $c=0,\dots, d-1$ and $\reg_0\otimes \rho$ for $\rho\in \irr(G)$ with $\dim\rho=2$. All these graded pieces are contained in $\cS$, except for $\reg_0\otimes \rho(x^{n^-})$. Since $\reg_\xi(A)\in \cB_\tau(A)$ and $\reg_{\eta}(A)\in \pi^*(\D(V/G))(A)$, we have $K(A)\in \cT$. Hence, also $\reg_0\otimes \rho(x^{n^-})\in \cT$.

Let $c=1,\dots,d-1$. We now know that all the graded pieces $\reg_0\otimes W$, except for $W=\chi(A(xy)^c)$, of the filtration of $\Phi_{\gamma}(\reg_{\pi_{\gamma}(0)})(A(xy)^c)\in \cA_\gamma(A(xy)^c)\subset \cT$ given in \autoref{lem:Phigammaimage} lie in $\cT$. Hence,  also $\reg_0(A(xy)^c)\in \cT$.

Now, all graded pieces of the filtration of $\reg_{\eta}(A)\in \pi^*(\D(V/G))(A)$ described in \autoref{lem:pifibres} are known to lie in $\cT$ except for $\reg_0(A)$, which then accordingly also lies in $\cT$.
 \end{proof}

\subsection{The Exceptional Collection - odd case}\label{subsect:exceptionalodd}
Just as in the case that $e$ is even, all the members of the spanning class $\cS$ of the right orthogonal complement $\cC$ of \eqref{eq:sodABodd} are exceptional objects; see \autoref{cor:skyexc}. However, to be able to organise them into an exceptional sequence we have to remove those which correspond to the bottom row of \eqref{eq:McKayodd}. As the top row and the upper right diagonal of \eqref{eq:McKayodd} correspond to objects isomorphic to the removed ones, we are left exactly with the objects corresponding to nodes of \eqref{eq:McKayodd} in the usual text color black.

We can organise them into an exceptional sequence by subsequently following the diagonals in \eqref{eq:McKayodd} along the arrows:
\begin{equation}\label{eq:skyexcodd}
\begin{split}
& \reg_0(xy), \reg_0\otimes \rho(x^2y), \reg_0(x^2y^2), \reg_0\otimes \rho(x^3y), \reg_0\otimes\rho(x^3y^2), \reg_0(x^3y^3),\dots,\\ & \reg_0\otimes \rho(x^{d-1}y), \dots, \reg_0(x^{d-1}y^{d-1}),\reg_0\otimes \rho(x^dy),\dots, \reg_0\otimes \rho(x^dy^{d-1}), \dots, \\&\reg_0\otimes \rho(x^{n^+-1}y),\dots, \reg_0\otimes \rho(x^{n^+-1}y^{d-1}) \,.
\end{split}
\end{equation}
We can again turn this into a full exceptional sequence of $\cC$ by adding the sheaves $F(x^a)$ of \autoref{def:F}.

\begin{prop}[$e$ odd]\label{prop:Cseqodd} The following is a full exceptional sequence of $\cC$:
\begin{equation}\label{eq:Cesodd}
\bigl(\text{Exceptional Sequence \eqref{eq:skyexcodd}} \bigr), F(x^0), F(x), \dots, F(x^{n^--1}), F(x^{n^-+1}),\dots,  F(x^{n^+-1})\,.
\end{equation}
\end{prop}
Note that all $F(x^a)$ with $a=0,\dots, n^+-1$ occur except for $F(x^{n^-})$.
\begin{proof}
This is analogous to the proof of \autoref{prop:Cseqeven}. One concludes from \autoref{lem:FExt}, that the $F(x^a)$ for $a=0,\dots,n^{+}-1$ are exceptional objects, semi-orthogonal to the members of \eqref{eq:skyexcodd}.
Furthermore, $\Ext^*_G\bigl(F(x^a), F(x^b)  \bigr)\neq 0$ if and only if $b=a+1$ or $(a,b)=(n^+-1, n^-)$. The latter extra case comes from the isomorphism $\rho(x^{n^+})\cong \rho(x^{n^-})$; see the comment below the formula in \autoref{lem:FExt}(ii). However, $F(x^{n^-})$ does not occur in \eqref{eq:Cesodd}, so the extra case does not disturb the semi-orthogonality.
\end{proof}

\begin{proof}[Proof of \autoref{thm:main} for $e$ odd]
Combining \autoref{prop:sodABodd} and \autoref{prop:Cseqodd}, we get the semi-orthogonal decomposition
\[
\D_G(V)=
\bigl\langle \text{(exc.\ seq.\ \eqref{eq:Cesodd})},\cA_\gamma(A(xy)), \dots, \cA_\gamma (A(xy)^{d-1}), \pi^*(\D(V/G))(A), \cB_\tau(A) \bigr\rangle\,.
\]
Comparing this with \autoref{lem:conjodd} confirms \autoref{conj:PvdB} for $G(m,e,2)$ with $e$ odd.
\end{proof}

\section{Final Comments}

\subsection{Linearity over $V/G$}\label{subsect:linear}
In this paper, we fixed a set of representatives of the conjugacy classes of $G=G(m,e,2)$ and then constructed for every representative $g$ a fully faithful embedding $\Theta_g\colon \D(V^{(g)})\hookrightarrow \D_G(V)$. For $g$ a reflection, these embeddings where denoted by $\Phi_g$ or $\Psi_g$, and for $g=\id$, we found $\Theta_{\id}=\pi^*$. Furthermore, the members of the exceptional sequences \eqref{eq:Ceseseven} and \eqref{eq:Cesodd} are in bijection with the representatives of the conjugacy classes with $V^g=\{0\}=\Spec\IK$. It is unclear whether there is a canonical choice of this bijection. But making any choice of a bijection $g\mapsto E_g$,  we can consider the fully faithful embeddings $\Theta_g\colon \D(\Spec\IK)\to \D_G(V)$ given by $\Theta_g(W^*)=E_g\otimes W^*$ for every finite-dimensional graded $\IK$-vector space $W^*$.

All these $\Theta_g$ are Fourier--Mukai transforms relative over $\D(V/G)$ with respect to the quotient morphism $\pi\colon V\to V/G$ and the morphism $\beta_g\colon V^{(g)}\to V/G$ induced by the $\CC(g)$-invariant composition
\[
 V^{g}\hookrightarrow V\xrightarrow\pi V/G\,.
\]
Note that, for $g$ a reflection, $\beta_g$ is the same as the canonical morphism $\nu_g\colon V^{(g)}\to Z_g/G$ used earlier followed by the embedding $Z_g/G\hookrightarrow V/G$.

This means that the Fourier--Mukai kernels of $\Theta_g$, which are objects in $\D_G(V^{(g)}\times V)$, arise as push forwards of objects in the equivariant derived category $\D_G(V^{(g)}\times_{V/G} V)$ of the fibre product.
It implies that the functors are $\D(V/G)$-linear in the sense that
\[
\Theta_g(A\otimes \alpha_g^*(B))\cong \Theta_g(A)\otimes \pi^*(B)\quad\text{ for $A\in \D(V^{(g)})$, $B\in \D(V/G)$;}
\]
see \cite[Lem.\ 2.35]{Kuz}. Hence, our semi-orthogonal decompositions are preserved under tensor products by objects of $\pi^*\D(V/G)$.
\subsection{Non Algebraically Closed Ground Fields}\label{subsect:groundfield}
For the whole paper, we had the assumption that $\IK$ is of characteristic zero and algebraically closed. For some steps of the proofs, $\IK$ being algebraically closed might indeed be necessary (or not, the author has not this checked in detail).

However, if we fix the group $G=G(m,e,2)$, it is enough to assume that $\IK$ is of characteristic zero and contains all $m$-th roots of unity for \autoref{thm:main} to hold. The condition that $\IK$ contains the roots of unity is necessary anyway for $G(m,e,2)\le \GL(\IK^2)$ to be well-defined.

Indeed, we can deduce the statement for the more general ground field $\IK$ from the case of its algebraic closure $\overline \IK$. The reason is that all the functors $\Theta_g\colon \D(V^{(g)})\to \D_G(V)$ constructed are already defined over $\IK$. Note, in particular, that the description of $\irr(G)$ given in \autoref{sect:G} holds already over $\IK$, so all the sheaves $\reg_0\otimes W$, which where proved to be exceptional over $\overline \IK$ can already be defined over $\IK$. Now, one can apply the descent result \cite[Thm.\ B]{BS--descent} to the faithfully flat morphism $S'=\Spec \overline \IK\to S=\Spec \IK$.

\appendix{
\newgeometry{left=20mm,right=20mm,top=2.5cm,bottom=1.5cm}
\begin{landscape}
\section{McKay Quiver - even case}\label{app:even}
Diagram \eqref{eq:evenMcKay} below is a diagrammatic description of \autoref{prop:skyExt} in the case that $e$ is even.
After removing the top row, whose nodes are displayed in tiny font size, the nodes of the diagram, labeled by polynomials $f\in \IK[x,y]$ are in bijection with the isomorphism classes of irreducible $G$-representations; see \autoref{prop:irreven}. More concretely, a polynomial $f$ in one of the four outer columns (two on the left and two on the right) is semi-invariant, and corresponds to the $G$-character $\chi(f)$. The monomials $x^ay^b$ in the inner columns correspond to the 2-dimensional irreducible representations $\rho(x^ay^b)$.

To interpret the diagram in terms of \autoref{prop:skyExt} one further sends $W\in \irr(G)$ to $\reg_0\otimes W\in \Coh_G(V)\subset \D_G(V)$. More concretely, the semi-invariant $f$ in the outer columns stand for $\reg_0(f)$ and the $x^ay^b$ in the middle columns stand for $\reg_0\otimes\rho(x^ay^b)$. A solid arrow than stands for a non-vanishing $\Ext^1_G$ class between the two connected objects, in other words, a morphisms of degree 1 in $\D_G(V)$. A dotted arrow stands for a non-vanishing $\Ext^2_G$ class, i.e.\ a degree 2 morphism in $\D_G(V)$.

If we ignore the dotted arrows, the diagram is also known as the McKay quiver of $G$. This means that the number of arrows from a node corresponding to $U\in \irr(G)$ to a node corresponding to $W\in \irr(G)$ is $\dim\Hom_G(U, V\otimes W)$ where $V\cong \IK^2$ is the two dimensional representation which turns $G$ into a reflection group. This agrees with the $\Ext^1_G$ interpretation by the second formula of \autoref{lem:skyExt}.

The top row is redundant in the sense that the corresponding $W\in \irr(G)$ are isomorphic to those in the bottom row. However, we felt that including the top row improves readability, because it avoids arrows going over the whole diagram from top to bottom.
That the bottom and top row are in the color \textcolor{teal}{teal} is mainly to emphasize this redundancy. However, it also might help the reader with the orientation in \autoref{subsect:exceptionaleven}, as the $\reg_0\otimes W$ corresponding to the bottom row are those not appearing in the exceptional sequence \eqref{eq:skyexc}. The \textcolor{violet}{violet} nodes (leftmost column and $N^+$, $N^-$) are exactly the $\reg_0\otimes W$ not contained in $\cC$; see \autoref{prop:Bspanningeven}.

\begin{equation}\label{eq:evenMcKay}
\begin{tikzcd}
\scriptscriptstyle{\color{violet}{A(xy)^{d}\cong A}} \arrow[rr, bend right]               & \scriptscriptstyle{\color{teal}{(xy)^{d}\cong 1}} \arrow[r]               & \scriptscriptstyle{\color{teal}{x^{d+1}y^{d}\cong x}} \arrow[r]                                     & \scriptscriptstyle{\color{teal}{x^{d+2}y^{d}\cong x^2}}  \arrow[r]                     & \dots \arrow[r]            & \scriptscriptstyle{\color{teal}{x^{n+d-1}y^{d}\cong x^{n-1}}} \arrow[r]                                                     & \scriptscriptstyle{\color{violet}{N^+(xy)^{d}\cong N^+}} \arrow[r]               & \scriptscriptstyle{\color{violet}{N^-(xy)^{d}\cong N^-}}                            \\
\color{violet}{A(xy)^{d-1}} \arrow[rr, bend right]  \arrow[ru, dotted]             & (xy)^{d-1} \arrow[r]    \arrow[lu, dotted]           & x^dy^{d-1} \arrow[r]   \arrow[u, dotted] \arrow[llu] \arrow[lu]                                   & x^{d+1}y^{d-1} \arrow[r]    \arrow[u, dotted]      \arrow[lu]           & \dots \arrow[r]  \arrow[u, dotted] \arrow[lu]         & x^{n+d-2}y^{d-1} \arrow[r]    \arrow[u, dotted] \arrow[lu]                                                & N^+(xy)^{d-1} \arrow[r]   \arrow[ru, dotted]    \arrow[lu]         & N^-(xy)^{d-1}   \arrow[lu, dotted]           \arrow[llu]               \\
\color{violet}{\vdots} \arrow[ru, dotted]                        & \vdots \arrow[lu, dotted]          & \vdots \arrow[lu] \arrow[llu] \arrow[u, dotted]          & \vdots \arrow[u, dotted]                     &                            & \vdots \arrow[u, dotted]                                                       & \vdots \arrow[ru, dotted]             & \vdots \arrow[lu, dotted]               \\
\color{violet}{A(xy)^2} \arrow[rr, bend right] \arrow[ru, dotted] & (xy)^2 \arrow[r] \arrow[lu, dotted] & x^3y^2 \arrow[r] \arrow[u, dotted] \arrow[llu] \arrow[lu] & x^4y^2 \arrow[r] \arrow[u, dotted] \arrow[lu]  & \dots \arrow[r] \arrow[lu] & x^{n+1}y^2 \arrow[r] \arrow[rr, bend right] \arrow[u, dotted]  & N^+(xy)^2 \arrow[lu] \arrow[ru, dotted] & N^-(xy)^2 \arrow[llu] \arrow[lu, dotted] \\
\color{violet}{Axy} \arrow[rr, bend right] \arrow[ru, dotted]    & xy \arrow[r] \arrow[lu, dotted]    & x^2y \arrow[r] \arrow[u, dotted] \arrow[llu] \arrow[lu]   & x^3y \arrow[r] \arrow[lu] \arrow[u, dotted] & \dots \arrow[r] \arrow[lu] & x^{n}y \arrow[r] \arrow[rr, bend right] \arrow[u, dotted] \arrow[lu]           & N^+(xy) \arrow[ru, dotted] \arrow[lu]  & N^-(xy) \arrow[lu, dotted] \arrow[llu]   \\
\color{violet}{A} \arrow[rr, bend right] \arrow[ru, dotted]      & \color{teal}{1} \arrow[r] \arrow[lu, dotted]     & \color{teal}{x} \arrow[r] \arrow[u, dotted] \arrow[lu] \arrow[llu]     & \color{teal}{x^2}\arrow[r] \arrow[lu] \arrow[u, dotted]    & \dots \arrow[r] \arrow[lu] & \color{teal}{x^{n-1}} \arrow[r] \arrow[rr, bend right] \arrow[u, dotted] \arrow[lu]          & \color{violet}{N^+} \arrow[ru, dotted] \arrow[lu]     & \color{violet}{N^-} \arrow[lu, dotted] \arrow[llu]
\end{tikzcd}
\end{equation}
\end{landscape}

\begin{landscape}
\section{McKay Quiver - odd case}\label{app:odd}
Diagram \eqref{eq:McKayodd} below is a diagrammatic description of \autoref{prop:skyExt} in the case that $e$ is odd. The interpretation is the same as for \eqref{eq:evenMcKay} above. Note that the column with the $A(xy)^c$ on the left is missing. This is simply to be able to fit the diagram onto the page. In reality, the left side of the McKay quivers, standing for the characters $\chi(A(xy)^c)$ and $\chi((xy)^c)$ look exactly the same for $e$ even and $e$ odd, but there are differences on the right side, where we do not have the additional characters $\chi(N^{\pm}(xy)^c)$ anymore in the odd case.

Instead the rightmost diagonal, displayed in tiny font size and \textcolor{orange}{orange}, is identified with the right part of the bottom row via the isomorphisms $\rho(x^{n^+}y^c)\cong \rho(x^{n^-+c})$, which is the special case $\lambda=\frac{e+1}2$ of \autoref{lem:camparerhos}. There is still, similar to the even case, the identification of the top row with the left part of the bottom row, indicated by the color \textcolor{teal}{teal}.

The knot $x^{n^-}$ is \textcolor{violet}{violet} to indicate that $\reg_0\otimes \rho(x^{n^-})$ is not contained in $\cC$; see \autoref{prop:Bspanningodd}.

\begin{equation}\label{eq:McKayodd}
\begin{tikzcd}[column sep = small]
\scriptscriptstyle{\color{teal}{(xy)^{d}\cong 1}}  \arrow[r]                            & \scriptscriptstyle{\color{teal}{x^{d+1}y^{d}\cong x}} \arrow[r]                                              & \color{teal}{\dots} \arrow[r]            & \scriptscriptstyle{\color{violet}{x^{n^+}y^{d}\cong x^{n^-}}}                                &  &                                                                          &                                                       &                                                          &                                                        &                                                     \\
(xy)^{d-1} \arrow[r]                               & x^{d}y^{d-1} \arrow[r] \arrow[lu]  \arrow[u, dotted]                                             & \dots \arrow[r] \arrow[lu]           & x^{n^+-1}y^{d-1} \arrow[r] \arrow[lu] \arrow[u, dotted]                               & \scriptscriptstyle{\color{orange}{x^{n^+}y^{d-1}\cong x^{n^+-1}}} \arrow[lu] &                                                                          &                                                       &                                                          &                                                        &                                                     \\
 (xy)^{d-2} \arrow[r]  & x^{d-1}y^{d-2} \arrow[r] \arrow[lu]  \arrow[u, dotted]  & \dots \arrow[r] \arrow[lu] & x^{n^-+d-2}y^{d-2} \arrow[r] \arrow[lu] \arrow[u, dotted] & x^{n^+-1}y^{d-2} \arrow[lu] \arrow[r] \arrow[u, dotted]       & \scriptscriptstyle{\color{orange}{x^{n^+}y^{d-2}\cong x^{n^+-2}}} \arrow[lu] &                                                       &                                                          &                                                        &                                                     \\
 \vdots                & \vdots \arrow[lu]  \arrow[u, dotted]                         & {}                         & \vdots \arrow[u, dotted]                                  & \vdots \arrow[u, dotted]                                                            & \ddots \arrow[lu] \arrow[u, dotted]                                      & \color{orange}{\ddots} \arrow[lu]                                     &                                                          &                                                        &                                                     \\
 (xy)^2 \arrow[r]      & x^3y^2 \arrow[r] \arrow[u, dotted]  \arrow[lu]           & \dots \arrow[r] \arrow[lu] & x^{n^-+2}y^2 \arrow[r] \arrow[u, dotted] \arrow[lu]       & x^{n^-+3}y^2 \arrow[r] \arrow[lu] \arrow[u, dotted]           & \dots \arrow[r]                                                          & x^{n^+-1}y^{2} \arrow[lu] \arrow[r] \arrow[u, dotted] & \scriptscriptstyle{\color{orange}{x^{n^+}y^2\cong x^{n^-+2}}} \arrow[lu] &                                                        &                                                     \\
 xy \arrow[r]         & x^2y \arrow[r] \arrow[u, dotted]  \arrow[lu]              & \dots \arrow[r] \arrow[lu] & x^{n^-+1}y \arrow[r] \arrow[u, dotted] \arrow[lu]         & x^{n^-+2}y \arrow[r] \arrow[lu] \arrow[u, dotted]             & \dots \arrow[r] \arrow[lu]                                               & x^{n^+-2}y \arrow[r] \arrow[lu] \arrow[u, dotted]     & x^{n^+-1}y \arrow[lu] \arrow[r] \arrow[u, dotted]        & \scriptscriptstyle{\color{orange}{x^{n^+}y\cong x^{n^-+1}}} \arrow[lu] &                                                     \\
 \color{teal}{1} \arrow[r]          & \color{teal}{x} \arrow[r] \arrow[u, dotted] \arrow[lu]                   & \color{teal}{\dots} \arrow[r] \arrow[lu] & \color{violet}{x^{n^-}} \arrow[r] \arrow[u, dotted] \arrow[lu]            & \color{orange}{x^{n^-+1}} \arrow[lu] \arrow[r] \arrow[u, dotted]              & \color{orange}{\dots} \arrow[r] \arrow[lu]                                               & \color{orange}{x^{n^+-3}} \arrow[r] \arrow[lu] \arrow[u, dotted]      & \color{orange}{x^{n^+-2}} \arrow[r] \arrow[lu]   \arrow[u, dotted]                        & \color{orange}{x^{n^+-1}} \arrow[lu] \arrow[r] \arrow[u, dotted]       & \scriptscriptstyle{\color{violet}{x^{n^+}\cong x^{n^-}}} \arrow[lu]
\end{tikzcd}
\end{equation}
\end{landscape}
}
\restoregeometry

\bibliographystyle{alpha}
\addcontentsline{toc}{chapter}{References}
\bibliography{references}

\end{document}